\documentclass[11pt, english]{article}

\usepackage[margin= 20mm,bottom=21mm, top=20mm]{geometry}

\usepackage[square,sort,comma,numbers]{natbib}
\usepackage{amsthm}
\usepackage{amsmath}
\usepackage{amssymb}
\usepackage{setspace}
\usepackage{mathtools}
\usepackage{graphicx}
\graphicspath{ {./images/} }
\usepackage[hidelinks]{hyperref}
\usepackage{cleveref}
\usepackage{hyperref}
\usepackage{enumitem}
\usepackage{framed}
\usepackage{subcaption}
\usepackage{xspace}
\usepackage{thmtools} 
\usepackage{thm-restate}

\usepackage{thmtools}
\usepackage{thm-restate}

\usepackage{floatrow}

\usepackage{tikz}
\usepackage{mathdots}
\usepackage{xcolor}
\usepackage{diagbox}
\usepackage{colortbl}
\usepackage[absolute,overlay]{textpos}

\usetikzlibrary{calc}
\usetikzlibrary{decorations.pathreplacing}
\usetikzlibrary{positioning,patterns}
\usetikzlibrary{arrows,shapes,positioning}
\usetikzlibrary{decorations.markings}

\tikzstyle{edge}=[very thick]
\definecolor{bostonuniversityred}{rgb}{0.8, 0.0, 0.0}
\definecolor{arsenic}{rgb}{0.23, 0.27, 0.29}
\tikzstyle{diredge}=[postaction={decorate,decoration={markings,
		mark=at position .95 with {\arrow[scale = 1]{stealth};}}}]
\tikzset{
    arrow/.style={decoration={markings, mark=at position 0.7 with
    {\fill(-0.09*#1,-0.03*#1) -- (0,0) -- (-0.09*#1,0.03*#1) -- cycle;}}, postaction={decorate}},
    arrow/.default=1
}
\tikzset{
    arow/.style={decoration={markings, mark=at position 1 with
    {\fill(-0.09*#1,-0.03*#1) -- (0,0) -- (-0.09*#1,0.03*#1) -- cycle;}}, postaction={decorate}},
    arow/.default=1
}
\tikzset{
    arrrow/.style={decoration={markings, mark=at position 0.9 with
    {\fill(-0.09*#1,-0.03*#1) -- (0,0) -- (-0.09*#1,0.03*#1) -- cycle;}}, postaction={decorate}},
    arow/.default=1
}

\newcommand{\fitellipsis}[2] 
{\draw [fill=white]let \p1=(#1), \p2=(#2), \n1={atan2(\y2-\y1,\x2-\x1)}, \n2={veclen(\y2-\y1,\x2-\x1)}
    in ($ (\p1)!0.5!(\p2) $) ellipse [ x radius=\n2/2+0cm, y radius=1.1cm, rotate=\n1];
}
\newcommand{\Fitellipsis}[2] 
{\draw [fill=white]let \p1=(#1), \p2=(#2), \n1={atan2(\y2-\y1,\x2-\x1)}, \n2={veclen(\y2-\y1,\x2-\x1)}
    in ($ (\p1)!0.5!(\p2) $) ellipse [ x radius=\n2/2+0cm, y radius=1.4cm, rotate=\n1];
}

\theoremstyle{plain}

\newtheorem*{thm*}{Theorem}
\newtheorem{thm}{Theorem}[section]
\Crefname{thm}{Theorem}{Theorems}

\newtheorem*{lem*}{Lemma}
\newtheorem{lem}[thm]{Lemma}
\Crefname{lem}{Lemma}{Lemmas}

\newtheorem*{claim*}{Claim}
\newtheorem{claim}[thm]{Claim}
\crefname{claim}{Claim}{Claims}
\Crefname{claim}{Claim}{Claims}

\newtheorem{prop}[thm]{Proposition}
\Crefname{prop}{Proposition}{Propositions}

\Crefname{remar}{Remark}{Remarks}

\newtheorem{cor}[thm]{Corollary}
\crefname{cor}{Corollary}{Corollaries}

\newtheorem{conj}[thm]{Conjecture}
\crefname{conj}{Conjecture}{Conjectures}

\Crefname{qn}{Question}{Questions}

\Crefname{obs}{Observation}{Observations}

\Crefname{ex}{Example}{Examples}

\theoremstyle{definition}

\Crefname{prob}{Problem}{Problems}

\newtheorem{defn}[thm]{Definition}
\Crefname{defn}{Definition}{Definitions}

\theoremstyle{remark}

\renewenvironment{proof}[1][]{\begin{trivlist}
\item[\hspace{\labelsep}{\bf\noindent Proof#1.\/}] }{\qed\end{trivlist}}

\newcommand{\remove}[1]{}

\newcommand{\eps}{\varepsilon}
\renewcommand{\P}{\mathbb{P}}

\title{\vspace{-1 cm}
Short proofs of rainbow matching results}
\date{}

\author{
David Munh\'a Correia\thanks{
Department of Mathematics, ETH, Z\"urich, Switzerland. Research supported in part by SNSF grant 200021\_196965.
\newline
\emph{Emails}: \textbf{\{david.munhacanascorreia, benjamin.sudakov\}@math.ethz.ch}.
}
\and
Alexey Pokrovskiy\thanks{Department of Mathematics, University College London, Gower Street, London WC1E 6BT, UK. 
\newline 
\emph{Email}: \textbf{a.pokrovskiy@ucl.ac.uk}.}
\and
Benny Sudakov\footnotemark[1]}

\begin{document} 
\maketitle
\begin{abstract}
A subgraph of an edge-coloured graph is called rainbow if all its edges have distinct colours.
The study of rainbow subgraphs goes back 
to the work of Euler on Latin
squares and has been the focus of extensive research ever since. Many conjectures in this area roughly say that ``every edge coloured graph of a certain type contains a rainbow matching using every colour''. In this paper we introduce a versatile  ``sampling trick'', which allows us to obtain short proofs of old results as well as to solve asymptotically some well known conjectures.
\begin{itemize}
\item We give a simple proof of Pokrovskiy's  asymptotic version of the Aharoni-Berger conjecture with greatly improved error term.
\item We give the first asymptotic proof of the ``non-bipartite'' Aharoni-Berger conjecture, solving two conjectures of Aharoni, Berger, Chudnovsky and Zerbib.
\item We give a very short asymptotic proof of Grinblat's  conjecture (first obtained by Clemens, Ehrenm\"uller, and Pokrovskiy). Furthermore, we obtain a new asymptotically tight  
bound for Grinblat's problem as a function of edge multiplicity of the corresponding multigraph.
\item We give the first asymptotic proof of a 30 year old conjecture of Alspach.
\end{itemize}
\end{abstract}

\section{Introduction}\label{main}
Research regarding rainbow matchings in graphs dates back to the work of Euler on various problems about transversals in Latin squares. A Latin square of order $n$ is an $n \times n$ array filled with $n$ different symbols, where no symbol appears in the same row or column more than once. A transversal
in a Latin square of order $n$ is a set of $m$ entries such that no two entries are in the same row, same column, or have the same symbol. A transversal is said to be \textit{full} if $m = n$ and \textit{partial} otherwise. Despite the fact that not every Latin square contains a full transversal, it is plausible to ask whether every Latin square contains a large partial transversal. Indeed, the celebrated conjecture of Ryser, Brualdi and Stein states that every Latin square contains a transversal which uses all but at most one symbol.
\begin{conj} [Ryser-Brualdi-Stein~\cite{ryser1967neuere, brualdi1991combinatorial, stein1975transversals}]\label{ryser}

Every Latin square of order $n$ contains a transversal of size $n-1$.
\end{conj}
There is a bijective correspondence between Latin squares of order $n$ and proper edge-colourings of the complete bipartite graph $K_{n,n}$ with $n$ colours. Indeed, let a Latin square $S$ have $\{1,2, \ldots, n\}$ as its set of symbols and let $S_{i,j}$ denote the symbol at the entry $(i,j)$. To $S$ we associate an edge-colouring of $K_{n,n}$ with the colours $\{1,2,\ldots,n\}$ by setting $V(K_{n,n}) = \{x_1, \ldots, x_n,y_1, \ldots, y_n\}$ and letting the edge between $x_i$ and $y_j$ receive colour $S_{i,j}$. Note that this colouring is proper, and moreover, each colour consists of a matching of size $n$. It is now easy to see that transversals of size $m$ in $S$ correspond to rainbow matchings of size $m$ in the coloured $K_{n,n}$. Therefore, the Ryser-Brualdi-Stein conjecture states - every properly edge-colouring of $K_{n,n}$ with $n$ colours has a rainbow matching of size $n-1$. 

The Ryser-Brualdi-Stein conjecture is just one thread of the research on rainbow matchings and rainbow subgraphs more broadly. There are many other interesting conjectures, some of them motivated by strengthening Ryser-Brualdi-Stein, others motivated by other branches of mathematics. In this paper we give improved results on a broad range of such conjectures. Very roughly, the prototypical problem that we study is of the form ``every $n$-edge-coloured graph of a certain type has a rainbow matching using every colour''. As an example, consider the following conjecture of Aharoni-Berger.
\begin{conj}[Aharoni and Berger,~\cite{aharoni2009rainbow}]\label{fullAB}
 Let $G$  be a properly edge-coloured bipartite multigraph with $n$ colours having at least $n+1$ edges of each colour. Then $G$ has a  rainbow matching using every colour. 
\end{conj}
The motivation for this conjecture is the Ryser-Brualdi-Stein conjecture which it strengthens (to see this, consider a properly coloured $K_{n,n}$ as in Conjecture~\ref{fullAB}; delete one colour to obtain a graph satisfying the Aharoni-Berger conjecture). Given the difficulty of the Ryser-Brualdi-Stein conjecture, much of the effort has been put into proving asymptotic versions of Conjecture~\ref{fullAB}. There are two natural approaches one can take in proving weakenings of this conjecture, which we will refer to as a \emph{weak asymptotic} and a \emph{strong asymptotic}.

The weak asymptotic asks for rainbow matchings which \emph{uses nearly all colours}.
\begin{quote}\textbf{Weak asymptotic:}  Let $G$  be a properly edge-coloured bipartite multigraph with $n$ colours having at least $n+1$ edges of each colour. Then $G$ has a  rainbow matching of size $n-o(n)$.  
\end{quote}
A weak asymptotic version of the Aharoni-Berger conjecture was proved by Barat-Gy\'arf\'as-Sarkozy who prove the above with error term $o(n)=\sqrt n$.  Their proof was very short and elegant, using the method developed by Woolbright for his result (see \cite{woolbright1978n}) on the Ryser-Brualdi-Stein conjecture. 

Having obtained the weak asymptotic, we would now like to improve the error term, preferably to $o(n)=0$, at which point the conjecture would be proven. Somewhat surprisingly, there has since been no improvement to the error term in Barat-Gy\'arf\'as-Sarkozy's result --- despite close ties to the Ryser-Bruldi-Stein conjecture, none of the progress on that conjecture generalises to the Aharoni-Berger multigraph setting.

Another direction is to prove \emph{qualitatively stronger} asymptotic results. For us ``strong asymptotic'' will mean a result of the following type, which guarantees matchings \emph{using all the colours in the graph}, at the cost of having slightly more edges of each colour. We will usually say that rainbow matchings using all the available colours are \emph{full}.
\begin{quote}
\textbf{Strong asymptotic:}  Let $G$  be a properly edge-coloured bipartite multigraph with $n$ colours having at least $n+o(n)$ edges of each colour. Then $G$ has a  rainbow matching using every colour.  
\end{quote} 
The reason we call the above statement as a ``strong'' asymptotic, is that it implies the previously mentioned weak asymptotic. Indeed suppose we have a properly edge-coloured bipartite multigraph $G$ with $n$ colours having at least $n+1$ edges of each colour. Delete $o(n)$ colours in order to obtain a new graph $G'$ with $n'=n-o(n)$ colours and each colour having $n'+o(n)+1$ edges. The strong asymptotic applies to this to give a rainbow matching using every colour. This gives a rainbow matching of size $n'=n-o(n)$ in the original graph. Moreover, note that we can choose which $o(n)$ colours we want to miss. This simple argument shows that the ``strong asymptotic with error term $o(n)$ implies the weak asymptotic with error term $o(n)$. 

It was believed that the strong asymptotic is fundamentally more difficult than the weak one. Indeed, it took much longer for the strong asymptotic to be proved, and the proof methods involved were considerably more difficult. It is easy to see that if there are $2n$ edges of each colour has a rainbow matching of size $n$. Indeed, if the largest matching $M$ in such a graph had size $\leq n-1$, then one of the $2n$ edges of the unused colour would be disjoint from $M$, and we could get a larger matching by adding it.
This simple bound has been successively improved by many authors.
Aharoni, Charbit, and Howard \cite{aharoni2015generalization} proved first that matchings of size $\lfloor 7n/4\rfloor$ are sufficient to guarantee a rainbow matching of size $n$. Kotlar and Ziv \cite{kotlar2014large} improved this to $\lfloor 5n/3\rfloor$. 
The third author then proved that $\phi n+o(n)$ is sufficient, where $\phi\approx 1.618$ is the Golden Ratio~\cite{pokrovskiy2015rainbow}.
Clemens and Ehrenm\"uller \cite{clemens2015improved}  showed that $3n/2+o(n)$ is sufficient. Aharoni, Kotlar, and Ziv \cite{aharoni2017representation}  showed that having $3n/2+1$ edges of each colour in an $n$-edge-coloured bipartite multigraph guarantees a rainbow matching of size  $n$. Finally, the strong asymptotic, as stated above, was proved by the third author in~\cite{pokrovskiy2018approximate}. This proof was much longer and more difficult than Barat-Gy\'arf\'as-Sarkozy's proof of the weak asymptotic. It also gave a considerably weaker error term.
 
Now, we've already seen that ``if the strong asymptotic is true, then the weak asymptotic is true''.
The main idea of this paper is a very short trick, that we call ``the sampling trick'', which allows one to prove the converse statement. This trick will allow us to prove results like ``suppose the weak asymptotic is true with $o(n)=n/f(n)$; then the strong asymptotic is true with $o(n)=3n/\sqrt{f(n)}$''.
Combining this with the Barat-Gy\'arf\'as-Sarkozy result, we obtain the strong asymptotic version of the Aharoni-Berger conjecture with a much improved error term.
\begin{thm}
\label{bipartiteABthm}
Let $G$  be a properly edge-coloured bipartite multigraph with $n$ colours having at least $n+n^{3/4}$ edges of each colour. Then $G$ has a  rainbow matching using every colour. 
\end{thm}
\noindent
As mentioned before, the original proof of the strong asymptotic was quite involved and the corresponding paper was more than 40 pages long. Our approach, in addition to giving a polynomial error term, vastly simplifies it (a full proof will now take less than two pages).

Our ``sampling trick'' is very versatile and applies to many other problems and conjectures. In all our applications, it allows us to either prove a strong asymptotic for the first time, or to greatly simplify an existing proof of the strong asymptotic.
\subsection*{Non-bipartite Aharoni-Berger}
Since the result of Pokrovskiy, several recent papers have considered variants and extensions of the Aharoni-Berger conjecture. Notably, it is natural to ask what happens when we no longer require $G$ to be bipartite in Conjecture \ref{fullAB}. 
\begin{conj} [Gao et. al, \cite{gao2017full}]\label{generalfullAB}
Let $G$ be an edge-coloured multigraph with $n$ colours such that each colour class is a matching of size $n+2$. Then, $G$ contains a rainbow matching of size $n$.
\end{conj}
Note that in this more general case, we require $n+2$ edges in each colour class. This can be seen by the simple example of a proper $3$-edge-colouring of the disjoint union of two $K_4$'s. Contrary to Conjecture \ref{fullAB}, both the strong and weak asymptotics are unsolved here. Despite this, there are several recent results relating to this problem which place an additional restriction on the edge-multiplicity of the graph. Keevash and Yepremyan~\cite{keevash2018rainbow} showed that for every function $k = \omega(1)$, there is a $\eps = o(1)$ such that if each colour has at least $(1+\eps)n$ edges and $G$ has edge-multiplicity at most $n/k$, then it contains a rainbow matching of size $n-k$ (this is a weak asymptotic version of the above conjecture with an additional multiplicity assumption). Similarly, Gao, Ramadurai, Wanless and Wormald~\cite{gao2017full} proved that if the edge-multiplicity is at most $\frac{\sqrt{n}}{\log^2 n}$, there is $\eps = o(1)$ such that each colour class being of size at least $(1+\eps)n$, implies that there is a full rainbow matching (this is a strong asymptotic version of the conjecture with an additional multiplicity assumption). 

Finally, a recent result of Aharoni, Berger, Chudnovsky and Zerbib~\cite{aharoni2019large} states that if each colour class has $n$ edges, there is a rainbow matching of size $\frac{2}{3}n-1$. They also explicitly conjectured both the strong and weak asymptotic versions of Conjecture~\ref{generalfullAB}. In this paper we prove both of these.

\begin{thm}\label{generalABthm}
For all sufficiently large $n$, any $n$-edge-coloured multigraph such that each colour class is a matching of size at least $n + 20n^{1-1/16}$ contains a full rainbow matching.
\end{thm}

\subsection*{Grinblat's Problem}
Another interesting problem involving rainbow matchings which has been studied is the following. Let an \textit{$(n,v)$-multigraph} be an $n$-edge-coloured multigraph in which the edges of each colour span a disjoint union of non-trivial cliques that have in total at least $v$ vertices. These can be seen as a generalisation of the type of edge-coloured multigraphs we mentioned before. In fact, note that Conjecture \ref{fullAB} is equivalent to the statement that every bipartite $(n,2n+2)$-multigraph contains a rainbow matching of size $n$. The question raised by Grinblat, which was originally made in the context of his study on algebras of sets but has recently been considered as a graph-theoretic problem, is to determine the minimal $v = v(n)$ such that every $(n,v)$-multigraph contains a rainbow matching of size $n$. He conjectured the following.
\begin{conj} [Grinblat, \cite{grinblat2002algebras}]\label{grinblat}
For all $n \geq 4$, $v(n) = 3n-2$.
\end{conj}
\noindent Indeed, note that a lower bound of $v(n) > 3n-3$ occurs because we can take a disjoint union of $n-1$ triangles, each repeated in every one of the $n$ colours. This multigraph has no matching of size $n$. It is worth observing here that Conjecture \ref{grinblat} differs from the Aharoni-Berger problem in the single fact that we allow each colour class to be a disjoint union of cliques which can now be larger than just an edge. In fact, it is easy to note that one can reduce the problem to when each monochromatic clique is either a $K_2$ or a $K_3$. Conjecture \ref{generalfullAB} states that if each clique is a $K_2$, then $2n+4$ vertices in each colour class are sufficient to guarantee a full rainbow matching. As demonstrated by the above example, this changes substantially when we allow monochromatic triangles.

In terms of results towards Grinblat's problem, all effort has gone into proving the strong asymptotic version of it. Much like in the Aharoni and Berger problem, an easy greedy argument gives an upper bound on $v(n)$, namely of $v(n) \leq 4n$. Grinblat~\cite{grinblat2015families} showed that $v(n) \leq 10n/3 + o(n)$. Nivasch and Omri~\cite{nivasch2017rainbow} showed that $v(n) \leq \frac{16}{5}n + O(1)$ and later, Clemens, Ehrenm\"{u}ller and Pokrovskiy~\cite{clemens2017sets} proved the strong asymptotic version of Conjecture \ref{grinblat} by showing that $\nu(n) = 3n+O \left(\sqrt{n}\right)$. For this problem we give a one-paragraph argument proving the weak asymptotic and use again the sampling trick to establish the strong asymptotic version of
 Conjecture \ref{grinblat}, which despite not improving upon the bound from~\cite{clemens2017sets}, will only take a page. 
\begin{thm}\label{grinblatthm}
$v(n) = 3n+O \left(n^{3/4}\right)$.
\end{thm}
\noindent Clemens, Ehrenm\"{u}ller and Pokrovskiy~\cite{clemens2017sets} further asked the question of what occurs to $v(n)$ when we restrict our $(n,v)$-multigraph to be a simple graph. Recently, Munh\'a Correia and Yepremyan~\cite{correia2020full} determined this asymptotically, showing that every $(n,2n+o(n))$-multigraph which is simple (and even with multiplicity at most $\sqrt{n}/ \log^2 n$) contains a rainbow matching using all the colours. On the other hand, one can construct a $(n,2n)$-multigraph which is simple and does not contain such a rainbow matching by considering the Cayley table of $\mathbb{Z}_n$ for even $n$, which shows that the error term $o(n)$ is indeed needed. This leads to the following very natural question: given the maximum edge multiplicity of a $(n,v)$-multigraph, what is the minimal $v$ ensuring that it contains a rainbow matching of size $n$? Our next result essentially answers this for multiplicities $\eps n$ for some small $\eps>0$.
\begin{thm}\label{simplegrinblatthm}
Every $\left(n,2n+2m+O \left(n/(\log n)^{1/4}\right)\right)$-multigraph with edge multiplicities at most $m$ contains a rainbow matching using all the colours.
\end{thm}
\noindent Note first that this greatly generalises the result in \cite{correia2020full}, showing that whenever multiplicity is $o(n)$, already the same bound $v = 2n+o(n)$ as for the Aharoni-Berger problem is enough to guarantee the desired rainbow matching. Moreover, in Section \ref{multgrinb}, we will construct $\left(n,2n+2\eps n+ O(\eps^2 n)\right)$-multigraphs with edge multiplicity at most $\eps n$ and no rainbow matching of size $n$. This shows that the dependence on $m$ in the above theorem is asymptotically tight for $m=\eps n$ with small $\eps>0$.

\subsection*{Alspach's conjecture}
Recall that a $2$-factor is a spanning subgraph of a graph in which every vertex has degree 2. Much like the Ryser-Brualdi-Stein conjecture is about finding rainbow matchings in $1$-factorizations, one can look for them in $2$-factorizations too. In 1988, Alspach made the following conjecture. 
\begin{conj} [Alspach, \cite{alspach1988problem}]\label{alspach}
Let $G$ be a $2d$-regular graph, edge-coloured so that each colour is a 2-factor. Then, there exists a rainbow matching using every colour.
\end{conj}
There are several motivations for this conjecture. Firstly, it implies the Ryser-Brualdi-Stein conjecture for \emph{symmetric} Latin squares which have the same symbol on the diagonal. To see this, consider a symmetric Latin square whose rows/columns are indexed by $1, \dots, n$. Suppose that the symbols are $1, \dots, n$ with the main diagonal consisting of only $1$'s. Consider a graph with vertices $x_1^-, x_1^+, \dots, x_n^-, x_n^+$ where for all $i\neq j$ and $s,t\in\{-,+\}$ we have the edge $x_i^sx_j^t$ and colour it by the $(i,j)$th entry of the Latin square. It is fairly easy to see that this is a $(n-1)$-edge-coloured, $2$-factorized graph in which a full rainbow matching gives a partial transversal of size $n-1$ in the Latin square. The above conjecture also strengthens problems of  Cacetta-Mardiyono~\cite{caccetta1992premature} and  Chung (see~\cite{kouider1989existence})  who asked whether Conjecture~\ref{alspach} is true when each colour class is a Hamilton cycle, rather than a general $2$-factor.

The state of previous research on Alspach's conjecture closely mirrors that of the problems previously discussed.
The weak asymptotic for Alspach's conjecture was proved by Anstee and Cacetta~\cite{anstee1998orthogonal} who showed that there is always a rainbow matching of size $d-d^{2/3}$. 
In contrast to this, it was not known  that a rainbow matching using every colour exists when we additionally assume that $|G|\geq (1+o(1))2d$. However, there has been a sequence of results finding full rainbow matchings when $|G|$ is significantly larger than $2d$.
A greedy argument proves the conjecture when we assume $|G|\geq 4d-3$. This was improved to $|G|\geq 4d-5$ by Alspach, Heinrich and Li~\cite{alspach1992orthogonal}, to $|G|\geq 3.32 d$ by Kouider and Sotteau~\cite{kouider1989existence}, to   $|G| \geq 3d-2$ by Stong~\cite{stong2002orthogonal}, and finally to $|G|\geq 2\sqrt 2d+4.5$ by Qu, Wang, and Yan~\cite{qu2015orthogonal}. Our sampling trick improves on all these results and establishes the strong asymptotic version of Alspach's conjecture.
\begin{thm}\label{alspachthm}
Let $G$ be a $2d$-regular graph which is edge-coloured so that every colour class is a 2-factor. If $G$ has at least $2d+d^{3/4+o(1)}$ vertices and $d$ is sufficiently large, then
it has a rainbow matching using every colour.
\end{thm}

\section{The sampling trick and first applications}
In this section, we will introduce the sampling trick and give three short applications. Our approach will allow us to find rainbow matchings using all the colours available when we know the existence of one which uses almost all the colours. Informally, the idea behind the trick is the following. Given an edge-coloured multigraph $G$, such that each colour class has large size and some specific structure, our goal is to find a full rainbow matching. We will then randomly choose a set $S \subseteq V(G)$ of vertices, by putting each vertex in $S$ independently with some appropriately chosen small probability $p$. This will imply that each colour class has most of its edges in $G-S$, but still relatively many edges in $G[S]$. In order to construct a full rainbow matching, we then find a rainbow matching inside $G-S$ which uses almost all the colours and then complete it, by greedily finding a rainbow matching inside $G[S]$ which uses the rest of the colours.

In order to use the sampling trick in each application, we will need a standard probabilistic concentration bound, which will always be the following (see, e.g., \cite{hoeffding}).
\begin{lem}\label{chernoff}
Let $X$ be the sum of independent random variables $X_1, \ldots, X_n$ such that each $0 \leq X_i \leq k $. Then, for all $0 < \eps < 1$,
$$\mathbb{P} \left( |X - \mathbb{E}[X]| > \eps \mathbb{E}[X] \right) \leq 2 e^{-\eps^2 \mathbb{E}[X]/3k^2}$$
\end{lem}
\subsection{The Grinblat problem}\label{grinblatsec}
We will first give a short proof of Theorem \ref{grinblatthm}. As indicated above, in order to apply our sampling trick, we first need a result which provides us with a rainbow matching that uses almost all the colours.
\begin{prop}\label{claimgrinblat}
Let $G$ be a $(n,3n)$-multigraph. Then, it contains a rainbow matching of size at least $n - \sqrt{n}$.
\end{prop}
\begin{proof}
Let $M$ be a maximal rainbow matching in $G$ and suppose for contradiction sake that $|M|=n-k$ with $k > \sqrt{n}$. Let $C_0, |C_0|=k$ denote the set of colours not used in $M$ and $V_0 := V \setminus V(M)$. Then $V_0$ contains no $C_0$-coloured edge, as this would contradict the maximality of $M$. Moreover, since $M$ covers $2n-2k$ vertices, each colour $c \in C_0$ has at least $n+2k$ vertices in $V_0$ belonging to its colour class and so, there is a $c$-edge from each such vertex to $V(M)$. Finally, since the colour class of $c$ is a disjoint union of non-trivial cliques, these edges must be pairwise disjoint, as otherwise, there would be a $c$-edge connecting their endpoints in $V_0$. Therefore, for each colour $c \in C_0$, there exists a $c$-coloured matching $M_c \subseteq E[V_0,V(M)]$ of size at least $n+2k$.

Note also that by the maximality of $M$, there is no edge $e \in M$ for which there is a $C_0$-coloured rainbow matching consisting of two edges from $e$ to $V_0$. Given this, it is easy to check that for each $e \in M$, there are at most two colours $c_1,c_2 \in C_0$ such that two edges of $M_{c_1}$ and two edges of $M_{c_2}$ intersect $e$. On the other hand, as every 
$M_c, c \in C_0$ has size at least $n+2k$ and $M$ has $n-k$ edges, there are at least $3k$ edges in $M$ which intersect two edges in $M_c$. This implies that
$2n>2|M| \geq 3k|C_0|=3k^2>3n$, a contradiction.
\end{proof}
\begin{proof} [ of Theorem \ref{grinblatthm}]
Let now $G$ be a $(n,3n+40n^{3/4})$-multigraph. We can assume that each of the monochromatic cliques in the graph are either a $K_3$ or a $K_2$, since every clique can be partitioned into disjoint edges and at most one triangle which cover the same set of vertices. 
For each colour $c$, let then $t_c$ denote the number of triangles in its colour class and $l_c$ the number of edges, so that $3t_c + 2l_c \geq 3n + 40n^{3/4}$. 

Let $S \subseteq V(G)$ be a random set obtained by choosing each vertex independently with probability $p = 2n^{-1/4}$.
For each colour $c$, let $c[S]$, $c[G\setminus S]$ denote the sets of colour $c$ edges contained in $S$ and $G\setminus S$ respectively. Let $|c[S]|$, $|c[G\setminus S]|$ be the number of non-isolated vertices in each graph. Let us calculate $\mathbb{E}[|c[S]|]$. Each $K_3$ contributes $3p^3+6p^2(1-p)$ to this expectation, whereas each $K_2$ contributes $2p^2$ to it. Thus $\mathbb{E}[|c[S]|]=t_c(3p^3+6p^2(1-p))+2l_cp^2 \geq p^2(3t_c+2l_c) \geq 3p^2 n = 12\sqrt{n}.$ Similarly, 
\begin{eqnarray*}
\mathbb{E}[|c[G\setminus S]|]&=&t_c(3(1-p)^3+6(1-p)^2p)+2l_c(1-p)^2\geq (3t_c+2l_c)(1-p)^3\geq (3n + 40n^{3/4})(1-p)^3\\
&\geq& (3n + 40n^{3/4})(1-3p)=(3n + 40n^{3/4})(1-6n^{-1/4}) \geq 3n +20n^{3/4}.
\end{eqnarray*}
Notice that these random variables are sums of independent $[0,3]$-valued random variables. Therefore, by Lemma~\ref{chernoff}, we have $\P(|c[S]|<10\sqrt n)\leq o(n^{-1})$ and $\P(|c[G\setminus S]|< 3n)\leq o(n^{-1})$. By the union bound, with positive probability none of these events happen for any of the colours. 

Thus there exists a set $S$ with $|c[S]|\geq 10\sqrt n$ and $|c[G\setminus S]|\geq 3n$ for all colours $c$. Then, each $c[S]$ has at least $5\sqrt n$ edges and, by Proposition \ref{claimgrinblat}, there is a rainbow matching $M$ in $G-S$ of size at least $n-\sqrt{n}$. Let $C_0$ denote the set of colours not used in $M$. Since now each colour class in $C_0$ has maximum degree two and more than $2\cdot 2 \cdot |C_0|=4|C_0|=4\sqrt{n}$ edges in $G[S]$, we can greedily find a rainbow matching $N \subseteq G[S]$ which uses all colours in $C_0$. As a result, $M \cup N$ is a full rainbow matching in $G$.
\end{proof}

\subsection{The Bipartite Aharoni-Berger problem}
Next we give a short proof of Theorem \ref{bipartiteABthm}. As the reader might already anticipate, we first need a result which gives us, in this setting, a rainbow matching using almost all the colours. 
This was obtained by Bar\'{a}t, Gy\'{a}rf\'{a}s and S\'{a}rk\"{o}zy \cite{barat2017rainbow}, using a short alternating paths argument which dates back to the result of Woolbright \cite{woolbright1978n} that every Latin square of order $n$ contains a transversal of size $n - \sqrt{n}$.
\begin{prop}\label{baratbipAB}
For all sufficiently large $n$, any $n$-edge-coloured bipartite multigraph in which each colour class is a matching of size at least $n$ contains a rainbow matching of size at least $n - \sqrt{n}$.
\end{prop}
\begin{proof}[ of Theorem \ref{bipartiteABthm}]
Let $G$ be an $n$-edge-coloured bipartite multigraph such that each colour class is a matching of size $n + 7n^{3/4}$. 
Let $S \subseteq V(G)$ be a subset obtained by choosing each vertex independently with probability $p = 2n^{-1/4}$. 
For each colour $c$,  let $c[S]$, $c[G\setminus S]$ denote the sets of colour $c$ edges contained in $S$ and $G\setminus S$ respectively.
Letting $e(c[S])$, $e(c[G\setminus S])$ denote the number of these edges, we have $\mathbb E(e(c[S]))=p^2(n + 7n^{3/4})\geq 4\sqrt n$ and $\mathbb E(e(c[G\setminus S]))=(1-p)^2(n + 7n^{3/4})\geq (1-2p)(n + 7n^{3/4})\geq
n+2n^{3/4}$. Therefore,   by Lemma \ref{chernoff} and a union bound over all colours, we have that with positive  probability, colours have $e(c[S])\geq 3\sqrt n$ and $e(c[G\setminus S])\geq n$. Fix a set $S$ satisfying this. 

By Proposition \ref{baratbipAB}, there is a rainbow matching $M$ in $G-S$ of size at least $n-\sqrt{n}$. Let $C_0$ denote the set of colours not used in $M$. Since each colour in $C_0$ is a matching and has more than $2\cdot |C_0|=2\sqrt{n}$ edges in $G[S]$, we can greedily find a rainbow matching $N \subseteq G[S]$ which uses all colours in $C_0$. As a result, $M \cup N$ is a full rainbow matching in $G$.
\end{proof}
\subsection{The Alspach problem}
The last short application of the sampling trick will be the proof of Theorem \ref{alspachthm}. Here we will use
the well known results, proved by using the so called R\"odl-nibble type arguments, which state that nearly regular uniform hypergraphs with small codegrees have almost perfect matchings. This will allow us to find in this setting a rainbow matching using almost all the colours and then use the sampling trick to complete the proof.

\begin{proof}[ of Theorem \ref{alspachthm}]Let $\alpha<0.1$ be an arbitrarily small constant and let $G$ be a $2d$-regular graph on $n \geq 2d + d^{3/4+\alpha}$ vertices where $d$ is sufficiently large in terms of $\alpha$. Suppose further that $G$ is $d$-edge-coloured so that each colour forms a $2$-factor in $G$. Since every vertex has degree at most $2$ in any given colour, deleting the vertices belonging to some arbitrary edge in the graph can destroy at most $4$ edges of that colour. Since the number of edges of every colour is $n$, we can assume that $n \leq 4d$, since otherwise we can get a full rainbow matching greedily. Let $S \subseteq V(G)$ be a subset obtained by choosing each vertex independently  with probability $p=1-\frac{2d}n$.

For each colour, let $c[S]$, $c[G\setminus S]$ denote the colour $c$ edges contained in $S$ and $G\setminus S$, respectively. We have that $\mathbb E(e(c[S]))=p^2n=(n-2d)^2n^{-1}\geq d^{3/2+2\alpha}n^{-1}\geq d^{1/2+2\alpha}/4$ and $\mathbb E(e(c[G \setminus S]))=(1-p)^2n=4d^2/n$. For every vertex we have $|N(v)\setminus S|=(1-p)2d=4d^2/n$.
Next we prove concentration of all these random variables. For $|N(v)\setminus S|$ this is immediate from Lemma~\ref{chernoff} (since $|N(v)\setminus S|$ is a sum of independent $\{0,1\}$-valued random variables), so we have that each vertex $v$ has $|N(v)\setminus S|=4d^2/n\pm K\sqrt{d \log d}$ for some constant $K$ with probability $1-o(n^{-1})$. 
To prove concentration of $c[S]$, $c[G\setminus S]$, notice that since the colour class of $c$ is a $2$-factor, we can partition its edges into two sets $c_1,c_2$ both having at least $n/3$ different connected components with each component being either an edge or a path of length two. Then, $e(c[S]) = e(c_1[S]) + e(c_2[S])$ and each $e(c_i[S])$ is the sum of independent $\{0,1,2\}$-valued random variables with $\mathbb{E}[e(c_i[S])] \geq p^2n/3 \geq d^{1/2+2\alpha}/12$. Therefore, by Lemma \ref{chernoff}, each $e(c_i[S])=\mathbb{E}[e(c_i[S])]\pm K\sqrt{d\log d}$ with probability 
$1-e^{-\Omega(K^2\log d)}$. By taking $K$ to be large enough, this implies that for each colour $e(c[S])\geq d^{1/2+\alpha}$ with probability $1-o(n^{-1})$. The same argument gives $e(c[G\setminus S])= 4d^2/n\pm K\sqrt{d \log d}$ with probability $1-o(n^{-1})$. By taking a union bound over all vertices/colours, we have that with positive probability every colour has $e(c[S])\geq d^{1/2+\alpha}$, $e(c[G\setminus S])= 4d^2/n\pm K\sqrt{d \log d}$, and  every vertex $v$ has $|N(v)\setminus S|=4d^2/n\pm K\sqrt{d \log d}$. Fix a set $S$ for which all of these happen. 

Define an auxiliary $3$-uniform hypergraph $\mathcal{H}$ on $N \leq 5d$ vertices which consists of all edges of the form $(x,y,c)$, where $x,y$ are vertices in $G-S$ such that the edge $xy$ has colour $c$. By the choice of $p$, every vertex in $\mathcal{H}$ has degree $D \pm K\sqrt{d \log d}$, where $D:= 4d^2/n \geq d$. Moreover, since each colour class in $G$ was a $2$-factor, it is easy to check that $\mathcal{H}$ has codegree at most $2$. Therefore, by the well-known result of Kostochka and R\"{o}dl \cite{kostochka} (which extended the work of Alon, Kim and Spencer \cite{AKS}) on nearly-perfect matchings in hypergraphs, $\mathcal{H}$ has a matching covering all but at most $O\big(N/D^{(1-\alpha)/2}\big) = O\big(d^{(1+\alpha)/2} \big)$ many vertices. Every edge of this matching has a vertex representing a different colour $c$ and all but at most
$O\big(d^{(1+\alpha)/2} \big)$ colours are covered. So, deleting the vertices corresponding to colours 
gives a rainbow matching $M$ in $G-S$ of size at least $d-O\big(d^{(1+\alpha)/2} \big)$. Let $C_0$ denote the set of colours not used in $M$. Since each colour in $C_0$ has maximum degree at most two and has at least $d^{1/2+\alpha} >2 \cdot 2 \cdot  |C_0|=4|C_0|=O\big(d^{(1+\alpha)/2} \big)$ edges in $G[S]$, we can greedily find a rainbow matching $N \subseteq G[S]$ which uses all colours in $C_0$, so that $M \cup N$ is a full rainbow matching in $G$.
\end{proof}
\section{Non-bipartite Aharoni-Berger}\label{ABsection} 
In this section, we will prove Theorem \ref{generalABthm}. As the reader might already anticipate, we will first prove a weak asymptotic result and then use the sampling trick to finish. We make no serious attempt to optimize our error terms.

\begin{thm}\label{weaknonbipAB}
For all sufficiently large $n$, any $n$-edge-coloured multigraph in which each colour forms a matching of size $n$ contains a rainbow matching of size $n - 20n^{7/8}$. 
\end{thm}

\noindent In order to prove the above proposition, let us first give some definitions and notation. As usual, the length of a path will be the number of edges in it. Given a matching $M$, we let $V(M)$ denote the vertices incident to some edge of the matching and for such a vertex $x$, we denote by $m(x)$ the vertex such that $x m(x)$ is an edge of $M$. Also, suppose $X_1, X_2, \ldots, X_r$ are sets in an edge-coloured graph, such that every $X_i$ is either a set of colours or a set of edges. Call a path $P=v_1v_2 \ldots v_{r+1}$ an $X_1 - X_2 - \ldots - X_r$ path  if for every $i$, the edge $v_i v_{i+1}$ has either a colour in $X_i$ or belongs to the set of edges $X_i$. Finally, given a rainbow matching $M$ and a set of colours $C'$, a vertex $v$ will be called \emph{$(C',k)$-switchable for $M$} if there are at least $5k^2$ many $C'-M-\ldots-C'-M$ rainbow paths of length at most $k$, which start at a vertex outside $M$, end at vertex $v$ and are colour and vertex-disjoint aside from the common last edge $v m(v)$.

\begin{proof}[ of Theorem \ref{weaknonbipAB}]
Let $G$ be a graph satisfying the assertion of the theorem, $M$ a maximal rainbow matching in $G$ and suppose it has size at most $n - 20n^{7/8}$. Let us denote the set of at least $20n^{7/8}$ colours not used in $M$ as $C_0$, $V_0$ the set of vertices not in $V(M)$ and set $k = n^{1/8}$. Hence $|C_0| \geq 20n/k$. Define now pairwise disjoint sets $V_1, V_2, \ldots \subseteq V(M)$  together with submatchings $M_j := \{x m(x) : x \in V_j\} \subseteq M$ in the following recursive manner. For each $j$, let $M_j$ be the set of edges in the matching $M' := M \setminus \bigcup_{l < j} M_l$ that have an endpoint $x$, which is $(C_0,k)$-switchable for $M'$ in the graph $G' := G - \bigcup_{1 \leq l < j} m(V_l)$. Let $V_j$ be the set of these endpoints (if both endpoints of some edge are switchable we fix one arbitrarily). Note that by definition, no two vertices in $\bigcup V_j$ are matched in $M$.

\begin{claim}\label{firstclaim}
If $j \leq k$, then there is no $C_0-M'-\ldots-M'-C_0$ rainbow path of length at most $k$ whose endpoints are in $V(G') \setminus V(M')$. In particular, there is no $C_0$-edge in $G'$ contained outside $V(M')$.
\end{claim}
\begin{proof}
Suppose such a path $P$ exists and let $u,v$ be its endpoints so that for some $0 \leq l_u,l_v < j$, we have $u \in V_{l_u}$ and $v \in V_{l_v}$. Note that by definition of $V_{l_u}$, there exist at least $5k^2$ many $M- C_0 - \ldots - M- C_0$ rainbow paths of length at most $k$ which start at the edge $u m(u)$, end at a vertex in $\bigcup_{l < l_u} V_l$ and are colour and vertex-disjoint (aside from the first edge $u m(u)$). Since $2|P|\leq 2k < 5k^2$, one of these paths, which we denote by $P^{u}_1$, is colour and vertex-disjoint to $P$ (aside from the vertex $u$). Similarly, since $2|P|+2|P^{u}_1|\leq 4k < 5k^2$, we can next find a $M- C_0 - \ldots -M- C_0$ rainbow path, which we denote $P^{v}_1$, of length at most $k$ which starts at the edge $v m(v)$, ends at a vertex in $\bigcup_{l < l_v} V_l$ and is colour and vertex-disjoint to the path $P^{u}_1P$ (aside from the vertex $v$). Note we can continue this process and since $2 \cdot (2j+1) \cdot k < 5k^2$, we can ultimately find a path $P^{u}_r \ldots P^{u}_1 P P^{v}_1 \ldots P^{v}_s$ (for some $r,s \leq j$) which is a $C_0 - M - \ldots - M - C_0$ rainbow path with both endpoints in $V_0$. Note that this contradicts the maximality of $M$, as one can substitute the edges of $M$ belonging to this path with the $C_0$-edges in this path in order to construct a larger rainbow matching. 
\end{proof}
Let us now fix $i \leq k-1$ to be such that $|M_{i+1}| \leq n/k$ (such $i$ clearly exists since $\sum_j |M_j| \leq n$) and set $M' := M \setminus \bigcup_{l \leq i} M_l$ as well as the graph $G' := G - \bigcup_{1 \leq l \leq i} m(V_l)$. There are then at most $2|M_{i+1}|$ many $(C_0,k)$-switchable vertices for $M'$ in this graph. For simplicity, let us refer to these vertices from now on as just \emph{switchable} vertices. We first delete from $G'$ every $C_0$-coloured edge intersecting $V(M_{i+1})$ - note that from this, each colour in $C_0$ loses at most $2|M_{i+1}| \leq 2n/k$ edges. Therefore, each such colour now  has at least $n - |G \setminus  G'| - 2n/k \geq |M'| + 20n/k - 2n/k \geq |M'| + 18 n/k$ edges in the graph $G'$ (since $|G\setminus G'|=|M\setminus M'|\leq n-20n/k-|M'|$). Finally, we also define an edge in $G'$ to be \emph{heavy} if it is repeated in at least $t := 5k^3$ many colours belonging to $C_0$. 

Now, let us first trivially note that there cannot be a vertex $w \in V(M') \setminus V(M_{i+1})$ which has at least $5k^2$ distinct $C_0$-neighbours outside $V(M')$. Indeed, if this were the case, then $m(w)$ would be switchable, contradicting $w \notin V(M_{i+1})$. Secondly, recall that all the $C_0$-coloured edges touching $V(M_{i+1})$ were previously deleted and moreover, that by Claim \ref{firstclaim}, there is no $C_0$-coloured edge contained outside $M'$. Therefore, since $|V(M')| \leq 2n$ and $|C_0| \geq 20n/k$, the fact that no such vertex $w$ can exist implies that there are at most $|V(M')| \cdot (5k^2) \cdot t \leq 10k^2 t n \leq |C_0| \cdot (k^3 t/2)$ many $C_0$-coloured edges which are not heavy and have an endpoint outside $M'$. Let us delete all these edges. Note then that at least half of the colours $c \in C_0$ are such that at most $k^3 t$ of their edges were deleted. Call these colours $C_0'$, so that we have $|C_0'| \geq |C_0|/2$. Note next the following consequence.
\begin{claim}\label{secondclaim}
For every colour $c \in C'_0$, there are at least $14n/k$ many vertex disjoint $c-M'- \ldots -c- M'$ paths in $G'$ of length at most $k-2$, which start outside $V(M')$, whose $c$-edges are heavy and who end at vertex $v$ which is incident to a $c$-edge which is not heavy.
\end{claim}
\begin{proof}
Consider $M_c \cup M'$, where $M_c$ is the matching of edges of colour $c$. This is a union of alternating paths/cycles between edges of $M_c$ and $M'$ such that in each path the number of edges from $M_c$ is at most one larger then the number of edges from $M'$. Because of the previous deletion of $C_0$-coloured edges and since $c \in C'_0$, $M_c$ has now size at least $|M'| + 18n/k - k^3 t \geq |M'| + 17n/k$, and so there are at least $17n/k$ many $c - M' - \ldots - M' - c$ vertex disjoint paths with both endpoints outside $M'$. Further, by disjointedness of these paths, at most $\frac{|M'|}{(k-1)/2-1} = 2|M'|/(k-3) < 3|M'|/k \leq 3n/k$ of them have size larger than $k-2$ and so, at least $14n/k$ of them have size at most $k-2$. Let $P$ be such a path. Note that $P$ cannot be such that all its $c$-edges are heavy. Indeed, if that were the case, since $t > k-2$, we could pick distinct colours in $C_0$ for those edges in order to produce a rainbow path which contradicts Claim \ref{firstclaim}. Further, note that the first $c$-edge of $P$, which has an endpoint outside $M'$, must be heavy - since otherwise, it would have been deleted just before the statement of this claim. Therefore, we are done since then $P$ must contain a subpath $P'$ of the form $c-M- \ldots - c - M$ which starts outside $V(M')$, whose $c$-edges are heavy and which ends at a vertex $v$ which is an endpoint of a $c$-edge that is not heavy. 
\end{proof}
Define now for each colour $c \in C'_0$, the set $V_c\subset V(M')$ to be the set of at least $14 n/k$ many vertices $v$ which are produced by the above claim. Our final objective is now to show that there must be a vertex $w \in V(M') \setminus V(M_{i+1})$ which has at least $5k^2+2$ distinct $C_0$-neighbours in one of the sets $V_c$. Indeed, note that if this holds, then by the above claim, there are at least $5k^2+2$ many $c-M' - \ldots - M'$ alternating vertex-disjoint paths $P_1,P_2, \ldots$ of length at most $k-2$, whose $c$-edges are heavy, which start outside $M'$, and whose ending vertices $v_1,v_2, \ldots$ are such that each $w v_l$ is a $C_0$-coloured edge. Let us remove from this collection two paths $P_i,P_j$ which possibly intersect $\{w,m(w)\}$. Notice that the edges $w v_l$ all have different colours, since they touch $w$. We can then use that $t \geq 5k^2 \cdot k$, to pick the $C_0$-colours for the heavy edges in all the paths, so that the new paths $P'_l := m(w) w v_l + P_l$ are colour-disjoint and of length at most $k$, thus implying that $m(w)$ is switchable, and so, contradicting $w \notin V(M_{i+1})$ (notice that each $P'_l$ is indeed a path because of the earlier removal of $P_i,P_j$ from the collection). 

To finish, suppose no such vertex $w$ exists. For each colour $c \in C_0'$, Claim \ref{secondclaim} implies that there are at least $14n/k$ many non-heavy $c$-edges which were not deleted, and thus contained in $V(M') \setminus V(M_{i+1})$, that have an endpoint in $V_c$. Define $H$ to be the directed multigraph formed by these $c$-edges for all $c \in C'_0$ and orienting them towards the vertex which belongs to $V_c$. For a vertex $v \in V(M')$, let $d^+_H(v),d^-_H(v)$ denote the out and in-degrees of it in $H$. Note in particular that by the properties of these edges, $d^{-}_H(v)$ is equal to the number of colours $c \in C'_0$ such that $v \in V_c$. Therefore,
$\sum_{c \in C'_0} \sum_{v \in V_c} d^-_H(v) = \sum_v d^-_H(v)^2$. Further, by convexity, we have
$$\sum_{v} d^-_H(v)^2 \geq \frac{1}{|V(M')|} \left(\sum_v d^-_H(v) \right)^2 \geq e (H )^2/2n \geq (14n/k \cdot |C'_0|)^2/2n \geq (900n^2/k^3) \cdot |C'_0| $$
and so, there exists a colour $c \in C'_0$ such that $\sum_{v \in V_c} d^-_H(v) \geq 900n^2/k^3$, which is then a lower bound for the number of $C_0$-coloured edges contained in $V(M') \setminus V(M_{i+1})$ which are not heavy and have an endpoint belonging to $V_c$. At the same time, since no vertex $w$ as described earlier can exist, it must be that there are at most $|V(M')| \cdot (5k^2+2) \cdot t <  12k^2 t n=60k^5 n$ of these edges, which is a contradiction since $900n^2/k^3 =900nk^5 >60k^5 n$. 
\end{proof}

\noindent We can now use the sampling trick to complete the proof.
\begin{proof}[ of Theorem \ref{generalABthm}]
Let $G$ be an $n$-edge-coloured multigraph such that each colour class is a matching of size $n + 20n^{15/16}$. 
Let $S \subseteq V(G)$ be a subset obtained by choosing each vertex independently with probability $p = 7n^{-1/16}$. 
For each colour $c$,  let $c[S]$, $c[G\setminus S]$ denote the sets of colour $c$ edges contained in $S$ and $G\setminus S$ respectively.
We have $\mathbb E(e(c[S]))=p^2(n + 20n^{15/16})\geq 49n^{7/8}$ and $\mathbb E(e(c[G\setminus S]))=(1-p)^2(n + 20n^{15/16})\geq (1-2p)(n + 20n^{15/16})\geq
n+2n^{15/16}$. Therefore,   by Lemma \ref{chernoff} and a union bound over all colours, we have that with positive  probability, colours have $e(c[S])\geq 40n^{7/8}$ and $e(c[G\setminus S])\geq n$. Fix a set $S$ satisfying this. 

By Proposition \ref{weaknonbipAB}, there is a rainbow matching $M$ in $G-S$ of size at least $n-20n^{7/8}$. Let $C_0$ denote the set of colours not used in $M$. Since each colour in $C_0$ is a matching and has more than $2\cdot |C_0|=40n^{7/8}$ edges in $G[S]$, we can greedily find a rainbow matching $N \subseteq G[S]$ which uses all colours in $C_0$. As a result, $M \cup N$ is a full rainbow matching in $G$.
\end{proof}
\section{Bounded multiplicity Grinblat problem}\label{multgrinb}
In this section, we will prove Theorem \ref{simplegrinblatthm}. As expected, our main focus here will be to prove the following weak asymptotic result. 
In order to state it, let us further define now a $(n,v,m)$-multigraph to be a $(n,v)$-multigraph with maximum edge-multiplicity at most $m$. 

\begin{thm}\label{weaksimplegrinblatthm}
For all sufficiently large $n$, every $(n,2n+2m + n^{3/4},m)$-multigraph contains a rainbow matching of size $n -\frac{1001n}{\sqrt{\log n}}$.
\end{thm} 
\noindent The proof of the strong asymptotic result then follows as an application of the sampling trick.
\begin{proof}[ of Theorem \ref{simplegrinblatthm}]
Let $G$ be a $\left(n,2n+2m+\frac{1500n}{(\log n)^{1/4}},m \right)$-multigraph. We can assume that each of its monochromatic cliques are either a $K_3$ or a $K_2$, since every clique can be partitioned into disjoint edges and at most one triangle which cover the same set of vertices. 
For each colour $c$, let then $t_c$ denote the number of triangles in its colour class and $l_c$ the number of edges, so that $3t_c + 2l_c \geq 2n+2m+\frac{1500n}{(\log n)^{1/4}}$. Let $S \subseteq V(G)$ be a random set obtained by choosing each vertex independently with probability $p = 100(\log n)^{-1/4}$.
For each colour $c$, let $c[S]$, $c[G\setminus S]$ denote the sets of colour $c$ edges contained in $S$ and $G\setminus S$ respectively. Let $|c[S]|$, $|c[G\setminus S]|$ be the number of non-isolated vertices in each graph. By standard considerations, much like those done in Section \ref{grinblatsec}, it is easy to show that $\P \left(|c[S]|< \frac{4004n}{\sqrt{\log n}} \right) \leq o(n^{-1})$ and $\P(|c[G\setminus S]|< 2n+2m+n^{3/4})\leq o(n^{-1})$ for each colour $c$. By the union bound, with positive probability none of these events happen for any of the colours. 

Thus there exists a set $S$ with $|c[S]|\geq \frac{4004n}{\sqrt{\log n}}$ and $|c[G\setminus S]|\geq 2n+2m+n^{3/4}$ for all colours $c$. By Theorem \ref{weaksimplegrinblatthm}, there is a rainbow matching $M$ in $G-S$ of size at least $n-\frac{1001n}{\sqrt{\log n}}$. Let $C_0$ denote the set of colours not used in $M$. Since each colour class in $C_0$ has maximum degree two and more than $4|C_0|$ edges in $G[S]$, we can greedily find a rainbow matching $N \subseteq G[S]$ which uses all colours in $C_0$. As a result, $M \cup N$ is a full rainbow matching in $G$.
\end{proof}
\subsection{A matching problem}
Before going into the proof of Theorem \ref{weaksimplegrinblatthm}, we will first need the following simpler result. In this section, we prove that one can always find a matching of the size $n$, although it might not be a rainbow one.
\begin{lem}\label{matching}
Let $G$ be a $(n,2n+2m,m)$-multigraph. Then, it contains a matching of size $n$.
\end{lem}
\begin{proof}
For contradiction sake, let $M$ be a maximal matching in $G$ and suppose that $|M| \leq  n-1$. We will let $V_0$ denote the set $V(G) \setminus V(M)$ and will denote the edges in $E[V_0,V(M)]$ as \emph{external}. First, we give the following claim, which will essentially allow us to forget about the various cliques which can appear in the colour class of $c$, and only consider edges. 
\begin{claim}\label{triangles}
Let $c \in C$ be a colour for which there exists a set $A \subseteq V(M)$, such that no two vertices in $A$ are matched by $M$ and such that there are no $c$-edges contained in $V_0 \cup A$. Then, there exist at least $2m+2$ pairwise disjoint $c$-edges in $E[V(M) \backslash (A \cup m(A)),V_0 \cup A]$.
\end{claim}
\begin{proof}
Define the set $S := V(M) \setminus A$ and note that there is no $c$-edge contained outside $S$. Moreover, the colour class of $c$ has at least $2|M|+2m+2-|S| = |A|+2m+2$ vertices outside $S$ and so, there is a $c$-edge from each such vertex to $S$. Finally, since the colour class of $c$ is a disjoint union of non-trivial cliques, these edges must be pairwise disjoint, as otherwise, there would be a $c$-edge connecting their endpoints outside $S$. Therefore, there exists a $c$-coloured matching in $E[S,V \setminus S]$ of size at least $|A|+2m+2$. Since at most $|A|$ edges of this matching intersect $m(A)$, we are done.
\end{proof}
Note that in particular, the above claim implies a contradiction when $|A| = |M|$. Therefore, the goal of the subsequent arguments is to construct such a set $A$. Let us then first recursively define sets $E_1, E_2, \ldots \subseteq M$ and $V'_1,V_1,V'_2, V_2, \ldots \subseteq V(M)$ in the following way.
\begin{itemize}
    \item $E_1$ is the set of edges in $M$ which have an endpoint incident to at least $m+1$ edges which go to $V_0$. We let $V'_1$ be the set of those endpoints and $V_1 = m(V'_1)$.
    \item Having defined the sets $E_1, \ldots, E_{i-1}$ and while $\bigcup_{j < i} E_j \neq M$, we define $E_{i}$ to be the set of edges in $M \setminus \bigcup_{j < i} E_j$ which have an endpoint incident to at least $m+1$ edges which go to $\bigcup_{j < i} V_j$. We then let $V'_{i}$ be the set of those endpoints and $V_{i} = m(V'_i)$. 
\end{itemize}
\noindent As a result of the above definitions, note the following claim.
\begin{claim}
For any two vertices $u,v \in \bigcup_{l \leq i} V_l$, there is a maximal matching $M'$ such that $u,v \notin V(M')$ and $V(M \Delta M') \subseteq V_0 \cup \bigcup_{l \leq i} (V_l \cup V'_l)$. 
\end{claim}
\begin{proof}
We prove this by induction on $i$. For $i = 0$, the statement is trivial by taking $M' = M$. Suppose then that $i \geq 1$ and the statement is true for all smaller values. We can then assume that $u \in V_i$ and $v \in V_j$ for some $j \leq i$. Now, by assumption, $m(u)$ has at least $m+1$ edges which go to $\bigcup_{j < i} V_j$. At most $m$ of these edges go to $v$ and so, there is an edge $m(u)u'$ with $v \neq u' \in \bigcup_{l < i} V_l$. Now, if $j < i$, then we apply the induction hypothesis with the vertices $v,u'$. This gives a maximal matching $M'$ avoiding $v, u'$ and containing the edge $m(u)u$. Replacing $m(u)u$ by $m(u)u'$ gives a new maximal matching satisfying the claim. 
If $j=i$, then with a similar argument as above, there is an edge $m(v)v'$ with $u' \neq v' \in \bigcup_{l < i} V_l$. We now apply the induction hypothesis with the vertices $v',u'$. This gives a maximal matching $M'$ avoiding $v', u'$ and containing the edges $m(u)u, m(v)v$. Replacing these two edges by $m(u)u', m(v)v'$ gives a new maximal matching satisfying the claim. 
\end{proof}
Now, suppose that at some $i$, we have $\bigcup_{1 \leq j < i} E_j \neq M$. It is easy to note that the above claim implies that there is no edge in $\bigcup_{0 \leq j < i} V_l$. Indeed, for any such edge $e=(u,v)$ there is a maximal matching $M'$ with $u, v \not \in V(M')$ and therefore $M'$ can be extended using $e$ . Thus, letting $A = \bigcup_{1 \leq j < i} V_l$, we can apply Claim \ref{triangles} to get that each colour has at least $2(m+1)$ pairwise disjoint $c$-edges in $E[V(M) \backslash (A \cup m(A)),V_0 \cup A]$. By averaging over the $\leq 2n$ vertices of $V(M) \backslash (A \cup m(A))$, there exists an edge $e \notin \bigcup_{1 \leq j < i} E_j$ with an endpoint incident to at least $m+1$ edges going to $A \cup V_0$. Therefore, by definition, $E_i \neq \emptyset$. To finish, there must then exist some $i $, such that $\bigcup_{1 \leq j < i} E_j = M$, which is a contradiction by Claim \ref{triangles} with $A = \bigcup_{1 \leq j < i} V_j$, so that $|A| = |M|$.
\end{proof}
As a corollary, note the following.
\begin{cor}\label{Mprimelemma}
Let $G$ be a $(n,2n+2m+n^{3/4},m)$-multigraph, $M$ be a rainbow matching in $G$ and $C_0$ the set of colours not used in it. Then, for any $N \subseteq M$, there are at least $n^{1/4}$ many edge-disjoint matchings of size $|C_0|+ |N| - n^{3/4}$ using colours in $C_0 \cup C(N)$ and edges which are not contained in $V(M \setminus N)$.
\end{cor}
\begin{proof}
Delete all edges contained in $V(M \setminus N)$ and all colours not in $C_0 \cup C(N)$ (in this proof when we delete edge from a triangle of some color we substitute triangle with one of its non-deleted edges so that every color class is still union of cliques). Note this produces a $(|C_0 \cup C(N)|,2|C_0 \cup C(N)|+2m+n^{3/4},m)$-multigraph $G'$. Now, the desired consequence follows by applying the previous lemma to $G'$ at least $n^{1/4}$ many times and after each iteration deleting the maximal matching found along with those colours which appear more than $\sqrt{n}$ (of which there are at most $\sqrt{n}$ many) many times in that matching.
\end{proof}
\subsection{Building blocks}
One can view Corollary \ref{Mprimelemma} as somewhat of a strong indicator for Theorem \ref{weaksimplegrinblatthm} since we might suspect from the previous section on the non-bipartite Aharoni-Berger problem, that having many large matchings, despite not being necessarily rainbow, should be favorable in some way. Indeed, this will be the key observation here. Before diving into that, we will need to make some preliminary considerations first.

In this section, we will give some new definitions, but the reader should keep in mind that some notation will carry over from Section \ref{ABsection}. Throughout this section, we will always be considering an underlying multigraph $G$ which is edge-coloured with $n$ colours so that it is locally $2$-bounded, that is, there is no vertex incident on more than two edges of the same colour. We also require, for simplicity, that the edges of each colour form a simple graph. Note in particular, that any $(n,v,m)$-multigraph is an example of this. Let us start with a definition.
\begin{defn}
Given a rainbow matching $M$ and a set of colours $C$, a \emph{$(C,t, r)$-block for $M$} is a pair $(B,M')$ where $B$ is a set of vertices and $M' \subseteq M$ such that the following hold.
\begin{enumerate}
    \item $B$ contains exactly one vertex $v \notin V(M)$ and $M' = \{x m(x) : x \in B \setminus \{v\}\}$.
    \item $|M'| \leq t$.
    \item For all vertices $x \in B$, there is a $M'-C-\ldots -M'-C$ path of length at most $r$, starting at $x$ and ending at $v$, whose $C$-edges are each repeated in at least $n^{1/10}$ many colours of $C$. 
\end{enumerate} 
We will usually refer to the block as just the set $B$. We define the set of colours of the block to be $C(B) := C(M')$. We also define, $M(B) := M'$ and $v_{B} := v$. Two blocks $B,B'$ are said to be \emph{disjoint} if the sets $V(M(B)) \cup \{v_{B}\}$ and $V(M(B')) \cup \{v_{B'}\}$ are disjoint. 
\end{defn}
\noindent Notice that in particular, it follows from the definition that for all vertices $v \notin V(M)$, the pair $(v,\emptyset)$ is a $(C,0,0)$-block for $M$. 
\noindent We will now introduce two ways of iteratively constructing blocks. The first is simple to check, and we thus omit its proof.
\begin{lem}\label{blockconst1}
Let $(B,M')$ be a $(C,t,r)$-block for $M$ and $v := v_B$. Let $w_1, w_2, \ldots w_k \in V(M \setminus M')$ and $z_1,z_2, \ldots z_k \in B$ be distinct vertices such that for each $i$, there are at least $n^{1/10}$ many $C$-colours repeated in the edge $m(w_i)z_i$. Then, $(B \cup \{w_i : i \leq k\},M' \cup \{w_im(w_i) : i \leq k\})$ is a $(C,t+k,r+2)$-block for $M$.
\end{lem}
\begin{lem}\label{blockconst2}
Let $(B,M')$ be a $(C,t,r)$-block for $M$ and $v := v_B$. Let $P_1, P_2, \ldots \subseteq B$ be vertex disjoint paths of the form $C - M'- \ldots - M'-C$ whose endpoints are vertices $w$ with either $w = v$ or $w \in V(M')$ and $m(w) \notin B$ and whose $C$-edges are each repeated in at least $n^{1/10}$ many colours of $C$. Then, for each $i$, there is a choice of an endpoint $w_i$ of $P_i$ with $w_i \in V(M')$ and $m(w_i) \notin B$ so that $\left(B \cup \{m(w_1),m(w_2), \ldots\},M' \right)$ is a $(C,t,r+\max_i |P_i|)$-block for $M$.
\end{lem}
\begin{proof}
Let $P_i$ be one of the paths and for simplicity, first assume that both its endpoints, say $u_1,u_2$, belong to $V(M')$ and $m(u_1),m(u_2) \notin B$. Write the path $P_i$ together with the vertices $m(u_1),m(u_2)$ added as $$P := m(u_1)u_1x_1y_1x_2 \ldots x_ky_ku_2m(u_2) .$$
To recall, we have that $P \setminus \{m(u_1),m(u_2)\} \subseteq B \setminus \{v\}$, each edge $x_iy_i$ is in $M'$ and each edge $y_ix_{i+1}$ is repeated in at least $n^{1/10}$ many $C$-colours (defining $y_0 = u_1$ and $x_{k+1} = u_2$). Now, since $u_1 \in B$, by the definition of a $(C,r,t)$-block, there is a path $Q = u_1m(u_1)x'_1y'_1x'_2 \ldots x'_ly'_lv$ of length at most $r$ such that each edge $x'_iy'_i$ is in $M'$ and each edge $y'_ix'_{i+1}$ is repeated in at least $n^{1/10}$ many $C$-colours. Let $z$ be the last vertex (in the direction $u_1 \rightarrow v$) of $Q$ such that $z \in Q \cap P$. Notice first that we must have that $z = y'_q \in Q$ for some $q$ (we cannot have $z=x'_q$ because $x'_q\in Q\cap P\implies y'_q\in Q\cap P$, since $x'_qy'_q$ is an edge of $M$  and both paths are $M$-alternating). Further, if $z = x_s \in P$ for some $s$, then note that the path $$m(u_2)u_2y_kx_ky_{k-1} \ldots y_szx'_{q+1}y'_{q+1} \ldots y'_lv$$ lets us choose $w_i := u_2$. If $z = y_s \in P$, then the path $$m(u_1)u_1x_1y_1x_2 \ldots x_szx'_{q+1}y'_{q+1} \ldots y'_lv$$ gives us the choice of $w_i := u_1$. Finally, in the case that one endpoint of $P_i$ is $v$, we choose $w_i$ to be the other endpoint. It is now simple to check that adding the vertices $m(w_i)$ to $B$ ensures that it is still a $(C,t,r+\max_i |P_i|)$-block for $M$.
\end{proof}
In the next lemma, we will look at families $\mathcal{F}$ of disjoint $(C,t,r)$-blocks for $M$. We then define $V(\mathcal{F}) := \bigcup_{B \in \mathcal{F}} B$, $M(\mathcal{F}) := \bigcup_{B \in \mathcal{F}} M(B)$, $C(\mathcal{F}) := \bigcup_{B \in \mathcal{F}} C(B)$ and say that the \emph{size of $\mathcal{F}$}, denoted as $|\mathcal{F}|$, is the sum of the sizes of the blocks in it, that is, $|V(\mathcal{F})|$. For two such families $\mathcal{F}_1,\mathcal{F}_2$, we write $\mathcal{F}_1 \preceq \mathcal{F}_2$ when for every block $B_1 \in \mathcal{F}_1$ there exists a block $B_2 \in \mathcal{F}_2$ such that $B_1 \subseteq B_2$. 

The main idea of the proof of Theorem \ref{weaksimplegrinblatthm} will be the following. Given a maximal rainbow matching with some large enough defect, we iteratively construct a family of disjoint blocks whose total size is growing from step to step. At each iteration, we will apply the lemma given below to the present family of disjoint blocks. The use of Corollary \ref{Mprimelemma} will allow us to avoid the first case of the lemma from happening and the properties of the blocks will exclude the second option. Thus we will be assured that one of the last two cases in the lemma below always holds. This, in turn, will allow us to construct a new family of disjoint blocks whose total size is relatively larger than the old one.  This eventually leads to a contradiction since at some point the total size of the block family becomes larger than the number of vertices in the graph.
\begin{lem}\label{stepprocedure}
Let $t \leq 6^{2\sqrt{\log n}}$, $M$ be a rainbow matching and $C$ a set of colours. Let $\mathcal{F}$ be a family of disjoint $(C,t,r)$-blocks for $M$ so that $V \setminus V(M) \subseteq V(\mathcal{F})$ and let $C' := C \cup C(\mathcal{F})$. Then, one of the following holds.
\begin{enumerate}
    \item There is no collection of at least $n^{1/5} \log n$ many edge-disjoint matchings of size at least $|M(\mathcal{F})|+ 1000n/ \sqrt{\log n}$ using colours in $C'$ and without edges contained in $V(M \setminus M(\mathcal{F}))$.
    \item There exists an edge $e = xy$ for which there are two distinct blocks $B_1,B_2 \in \mathcal{F}$ with $x \in B_1$, $y \in B_2$ and $c(e) \in C' \setminus (C(B_1) \cup C(B_2))$. 
    \item There exist at least $20n/\sqrt{\log n}$ many vertices $v \in V(M \setminus M(\mathcal{F}))$ with the following property: there are at least $n^{1/10}$ many distinct vertices $x \in V(\mathcal{F})$ such that the edges $m(v)x$ have distinct colours in $C'$ and further, the colour of each $m(v)x$ does not belong to the block in $\mathcal{F}$ which contains $x$.
    \item There is a family $\mathcal{F}'$ of disjoint $(C',3t+1,r+\sqrt{\log n}/100)$-blocks for $M$ with $\mathcal{F} \preceq \mathcal{F}'$ and $|\mathcal{F}'| \geq |\mathcal{F}| + 10n/\sqrt{\log n}$.
\end{enumerate}
\end{lem}
\begin{proof}
Let us first define the parameters $s := n^{1/5} \log n$ and $k = \frac{1000 n}{\sqrt{\log n}}$. Let $(B_i,M_i)$ denote the blocks in $\mathcal{F}$ and let $v_i := v_{B_i}$ for each $i$. Suppose that none of the first three options hold. In particular, from the first option, we are given $s$ edge-disjoint $C'$-coloured matchings $N_1, N_2, \ldots$ of size at least $|M(\mathcal{F})| + k$ and without edges contained in $V(M \setminus M(\mathcal{F}))$. We will use these to construct the family $\mathcal{F}'$. Before that, we will need two edge-deletion processes. First, let us delete, for each block $B_i \in \mathcal{F}$ and colour $c \in C(B_i)$, the $c$-coloured edges touching $B_i$. Since the graph is locally 2-bounded, each block $B_i$ has size at most $|B_i| \leq 2|M_{i}|+1 \leq 2t+1$ and every color appears in exactly one block (as they are disjoint), we delete at most $2(2t+1)$ edges of each color and in total at most $ 2(2t+1)n$ edges. In particular, then at most $4(2t+1)n/k < s/2$ matchings $N_i$ are such that at least $k/2$ of its edges were deleted. Let us from now on only consider the other $s/2$ matchings, implying that each of these has now at least $|M(\mathcal{F})|+k/2$ edges. This first deletion implies the following. 
\begin{claim}\label{claim1}
There is no $C'$-edge between two distinct blocks in $\mathcal{F}$.
\end{claim}
\begin{proof}
Note that since we are assuming that the second option in the statement does not hold, the occurrence of a $C'$-edge $e$ between two distinct blocks $B,B'$ can only be possible when $c(e) \in C(B) \cup C(B')$. However, note that this implies that the edge $e$ was deleted in the process described above.
\end{proof}
\noindent Next, for each $i$, let us define $M'_i \subseteq M_i$ to be those edges with both of its endpoints in $B_i$. As a second deletion process, delete all edges which touch vertices in the set $V(M(\mathcal{F})) \setminus V(\mathcal{F})$, which has size $\sum_i |M_i \setminus M'_i|$. Note then that each matching $N_j$ loses at most $\sum_i |M_i \setminus M'_i|$ of its edges, so that it has now at least 
$k/2 + \sum_i |M'_i|$
edges. This in turn implies the following standard claim, very similar to an earlier consideration done in Section \ref{ABsection}.
\begin{claim}\label{claim2}
For each $j$, there are at least $k/4$ many non-trivial vertex-disjoint $N_j -\bigcup_i M'_i - \ldots - \bigcup_i M'_i - N_j$ paths of length at most $10n/k$ and with both endpoints outside $\bigcup_i M'_i$.
\end{claim}
\begin{proof}
Fix some $j$. Consider $N_j \cup \bigcup_i M'_i$. This is a union of alternating paths/cycles between edges of $N_j$ and $\bigcup_i M'_i$ such that in each path, the number of edges from $N_j$ is at most one larger then the number of edges from $\bigcup_i M'_i$. Since $N_j$ has size at least $k/2 + \sum_i |M'_i|$, there are at least $k/2$ many $N_j -\bigcup_i M'_i - \ldots - \bigcup_i M'_i - N_j$ vertex disjoint paths with both endpoints outside $\bigcup_i M'_i$. Further, by disjointedness of these paths and since $\left|\bigcup_i M'_i \right| \leq |M| \leq n$, at most $\frac{n}{5n/k} < k/4$ of them have size larger than $10n/k$ (since each such path contains at least $5n/k$ edges of $\bigcup_i M'_i$) and so, at least $k/4$ of them have size at most $10n/k$. 
\end{proof}
\noindent Furthermore, notice that the second deletion process implies, together with Claim \ref{claim1}, that every edge in each $N_i$ is now either completely contained in some block in $\mathcal{F}$ or has one endpoint in $V(M \setminus M(\mathcal{F}))$ and the other in $V(\mathcal{F})$. Indeed,  recall first that by assumption, no edge in $N_i$ is contained in $V(M\setminus M(\mathcal{F}))$. Recall also that every vertex outside $M$ is contained in some block in $\mathcal{F}$ and so, the set of vertices outside $M \setminus M(\mathcal{F})$ which do not belong to any block is precisely equal to $V(M(\mathcal{F})) \setminus V(\mathcal{F})$. Therefore, since we deleted all edges touching $V(M(\mathcal{F})) \setminus V(\mathcal{F})$, if an edge of $N_i$ has one endpoint in $V(M \setminus M(\mathcal{F}))$, the other must be in $V(\mathcal{F})$; if the edge is entirely outside $M \setminus M(\mathcal{F})$, then both its endpoints belong to blocks in $\mathcal{F}$ and so, Claim \ref{claim1} implies that it is completely contained in some block. Finally, note that then, each path given by Claim \ref{claim2}, depending on whether or not its endpoints are contained in $V(M \setminus M(\mathcal{F}))$, is either completely contained in some block in $\mathcal{F}$ or such that one of its extremal edges, i.e., its first or last edge, has one endpoint in $V(M \setminus M(\mathcal{F}))$ and the other in $V(\mathcal{F})$.    
 
We can now describe the procedure which constructs the family $\mathcal{F}'$. First, we look at the case that for at least half of the $j$'s (and thus, at least $s/4$ of them), at least half of the paths (and thus, at least $k/8$ of them) given by Claim \ref{claim2} are completely contained in some block in $\mathcal{F}$. We claim that then there is some $j$ such that at least $k/16$ of these paths are such that their $N_j$-edges are repeated in at least $n^{1/10}$ colours in $C'$. Indeed, note that there are at most $2n \cdot \left(\max_j |B_j| \right)^{10n/k} \leq 2n \cdot (2t+1)^{10n/k} < \frac{k}{16} \cdot \frac{s}{4} / n^{1/10}$ non-trivial sequences of vertices entirely contained in some block of $\mathcal{F}$ and of length at most $10n/k$. Therefore, it must be that for some $j$, at least $k/16$ of the paths given by Claim \ref{claim2} are such that their implicit sequence of vertices is used for more than $n^{1/10}$ other values of $j$. In turn, since our graph is such that the edges of each colour form a simple graph, notice that this gives us the desired consequence. Now, let then $j$ be such that there are at least $k/16$ paths $P_1, P_2 \ldots$ given by Claim \ref{claim2} which are each entirely contained in some block of $\mathcal{F}$ and whose $C'$-edges are repeated in at least $n^{1/10}$ colours in $C'$. Take some $i$ and suppose $P^{(i)}_1, P^{(i)}_2, \ldots$ are those which are contained in the block $B_i$. Precisely, the properties of these paths (i.e., their endpoints are outside $\cup_i M'_i$) ensure that we can apply Lemma \ref{blockconst2} (to the block $(B_i,M_i)$) and add an endpoint of each of these paths to $B_i$ so that it becomes a $(C',t,r + 10n/k)$-block for $M$. Since all the paths are disjoint, we can do this to every block and therefore, construct the desired family $\mathcal{F'}$ - note indeed, that we end up adding at least $k/16$ vertices and so, $|\mathcal{F}'| \geq |\mathcal{F}| + k/16$.

Secondly, suppose that for at least $s/4$ of the $j$'s, at least $k/8$ of the paths given by Claim \ref{claim2} are such that one of its extremal edges, i.e., its first or last edge, has one endpoint in $V(M \setminus M(\mathcal{F}))$ and the other in $V(\mathcal{F})$. In particular, there are at least $k/8$ edges of $N_j$ with one endpoint in $V(M \setminus M(\mathcal{F}))$ and the other in $V(\mathcal{F})$. Define first the set $R$ to consist of those vertices $x \in V(M \setminus M(\mathcal{F}))$ with at least $n^{1/10}$ many distinct $C'$-neighbours in $V(\mathcal{F})$ - this will be a set of \textit{forbidden} vertices and we will increase it throughout the process. At the moment, the fact that Option 3 does not hold along with the first deletion process done at the start of the proof, gives that $|R| < k/32$. We also define first a set $B'_j := \emptyset$ for each $j$. Now, while $|R| < k/16$, we repeat the following operation. For each of the at least $s/4$ $j$'s, at least $k/8 - |R| > k/16$ of the edges given are disjoint to $R$. Therefore, by averaging, there exists a vertex $x \in V(M \setminus M(\mathcal{F})) \setminus R$ such that $m(x) \notin R$ is incident on at least $s/4 \cdot k/16 \cdot 1/2n = sk/128n$ edges of $\bigcup_j N_j$ which are disjoint to $R$. Note further that since $m(x) \notin R$, it must have at most $n^{1/10}$ many distinct $C'$-neighbours in $V(\mathcal{F})$. Therefore, recalling that our underlying edge-coloured multigraph is locally 2-bounded and the edges of each colour form a simple graph, there exists some vertex $y \in V(\mathcal{F})$ such that the pair $m(x)y$ is repeated in at least $sk/256n \cdot 1/n^{1/10} > n^{1/10}$ many $C'$-colours. Let $j$ be such that $y \in B_j$ and add the vertex $x$ to the set $B'_j$. Also, insert the vertices $x,m(x),y$ into $R$ in order to repeat the operation. Note that at the end of the process, we have sets $B'_j$ such that $\sum |B'_j| \geq k/96$, since at each iteration, three vertices are added to $R$ and one is added to $\bigcup B'_j$. Furthermore, notice that for each $j$, if at some iteration we add a vertex $x$ to $B'_j$, we also take a vertex $y \in B_j$ which was not in $R$ and insert it into $R$ - this implies that the number of vertices of $B_j$ contained in $R$ increases. Therefore, we cannot add more than $|B_j|$ vertices into $B'_j$ throughout the whole process, that is, $|B'_j| \leq |B_j|$ and so, by construction, Lemma \ref{blockconst1} ensures that $B_j \cup B'_j$ is a $(C',t+(2t+1),r+2)$-block for $M$. Note also that the family $\mathcal{F}'$ consisting of the blocks $B_j \cup B'_j$ is indeed a family of disjoint $(C',3t+1,r+2)$-blocks for $M$. This is the case since at each iteration of the process, whenever a vertex $x \in V(M \setminus M(\mathcal{F}))$ is inserted into $R$, the vertex $m(x)$ is also inserted. Therefore, in the end, there will never be an edge $x m(x) \in M$ such that the vertices $x, m(x)$ belong to different sets $B'_j$ and so, the resulting blocks $B_j \cup B'_j$ will precisely be disjoint blocks because the matchings $M(B_j \cup B'_j)$ are disjoint. Note finally that we also have $|\mathcal{F}'| \geq |\mathcal{F}| + \sum |B'_j| \geq |\mathcal{F}| + k/96$.
\end{proof}
\subsection{Proof of Theorem \ref{weaksimplegrinblatthm}}
We are now ready to prove our weak asymptotic result. Let $G$ be a $(n,2n+2m+n^{3/4},m)$-multigraph, $M$ be a maximal rainbow matching in $G$ and suppose, for contradiction sake, that $|M| < n - 1001n/ \sqrt{\log n}$. Let $V_0$ denote the set of vertices not used in $M$ and $C_0$ denote the set of colours not in $C(M)$. We will describe a process which will allow us to essentially cover the whole vertex set with small blocks and achieve a contradiction. The process goes as follows.

To start, let us set $M_0 := M$, $G_0 := G$ and $\mathcal{F}_0$ to be the collection of all singleton sets $\{v\}$ with $v \in V_0$, which is a family of disjoint $(\emptyset,0,0)$-blocks for $M_0$ in $G_0$ and has $|\mathcal{F}_0| = |V_0|$. We will refer to this as step 0. In general, the situation will be as follows after step $i \geq 0$ is completed. In our current graph $G_i$, we have a rainbow matching $M_i \subseteq M$, a set of colours $C_0 \subseteq C_{i-1}$ along with a family $\mathcal{F}_{i}$ of disjoint $(C_{i-1},4^i,i\sqrt{\log n}/100)$-blocks for $M_i$ in $G_i$, with $V(G_i) \setminus V(M_i) \subseteq V(\mathcal{F}_i)$ and $|\mathcal{F}_i| \geq \frac{10in}{\sqrt{\log n}} + |V_0|$. We will later show that provided that $i \leq \sqrt{\log n}$, we can apply Lemma \ref{stepprocedure} to the graph $G_i$ and the rainbow matching $M_i$, the family $\mathcal{F}_i$ and the set of colours $C_{i-1}$, and be assured that none of the first two options hold. This being the case, step $i+1$ will go as described next, depending on which option of the lemma holds. 

\vspace{0.2cm}
\noindent \textbf{If Option 3 holds -} then there exist at least $\frac{20n}{\sqrt{\log n}}$ vertices $v \in V(M_i \setminus M_i(\mathcal{F}_i))$ with the property that there are at least $n^{1/10}$ many distinct $x \in V(\mathcal{F}_i)$ such that the edges $m(v)x$ have distinct colours in $C_{i-1} \cup C(\mathcal{F}_i)$ and the colour of each $m(v)x$ does not belong to the block in $\mathcal{F}_i$ containing $x$. Let us take a subset $V'_{i+1}$ of these vertices so that $m(V'_{i+1}) \cap V'_{i+1} = \emptyset$ and $|V'_{i+1}| \geq \frac{10n}{\sqrt{\log n}}$. We then delete the vertices $m(V'_{i+1})$ to form the new graph $G_{i+1} := G_i - m(V'_{i+1})$. We also then take a new matching $M_{i+1} := M_i \setminus \{vm(v) : v \in V'_{i+1}\}$, define $C_i := C_{i-1} \cup C(\mathcal{F}_i)$ and define $\mathcal{F}_{i+1} := \mathcal{F}_i \cup \bigcup_{v \in V'_{i+1}} \{v\}$. Note that $\mathcal{F}_{i+1}$ is a family of $(C_{i-1},4^{i},i\sqrt{\log n}/100)$-blocks for $M_{i+1}$ in $G_{i+1}$ (in particular, it is a family of $(C_{i},4^{i+1},(i+1)\sqrt{\log n}/100)$-blocks) and has $V(G_{i+1}) \setminus V(M_{i+1}) \subseteq V(\mathcal{F}_{i+1})$ and $|\mathcal{F}_{i+1}| \geq |\mathcal{F}_i| + |V'_{i+1}| \geq \frac{10(i+1)n}{ \sqrt{\log n}} + |V_0|$. 
\vspace{0.2cm}

\noindent \textbf{If Option 4 holds -} then we let $C_i := C_{i-1} \cup C(\mathcal{F}_{i})$ and take a family $\mathcal{F}_{i} \preceq \mathcal{F}_{i+1}$ of disjoint $(C_i,4^{i+1}, (i+1)\sqrt{\log n}/100)$-blocks for $M_i$ in $G_{i}$ such that $|\mathcal{F}_{i+1}| \geq |\mathcal{F}_i| + \frac{10n}{\sqrt{\log n}} \geq \frac{10(i+1)n}{\sqrt{\log n}} + |V_0|$. We then take $M_{i+1} := M_i$, $G_{i+1} := G_i$ and note that $V(G_{i+1}) \setminus V(M_{i+1}) \subseteq V(\mathcal{F}_{i+1})$.
\vspace{0.2cm}

Now that the process is fully described, let us explain why showing that it is successful while $i \leq \sqrt{\log n}$ constitutes a contradiction, and thus, a proof of Theorem \ref{weaksimplegrinblatthm}. Indeed, note that if we are able to achieve step $i = \lfloor \sqrt{\log n} \rfloor$ and complete it, we have, as described above, that $|\mathcal{F}_{i}| \geq \frac{10in}{ \sqrt{\log n}} + |V_0| > |V(M)| + |V_0|$. In turn, the total number of vertices in $G$ is precisely $|V(M)| + |V_0|$, so we get a contradiction to $V(\mathcal F_i)\subseteq V(G)$. Let us now define the following property in the original graph $G$.

\vspace{0.2cm}
\noindent \textbf{Property P$_i$} - Let $U \subseteq V(\mathcal{F}_i)$ and $R \subseteq C_{i-1} \cup C(\mathcal{F}_i)$ be sets of size at most $n^{1/10}/(\log n)^{2i}$ such that no two members of $U \cup R$ belong to the same block in $\mathcal{F}_i$. Let also $B_1, B_2, \ldots \in \mathcal{F}_i$ be a collection of at most $n^{1/10}/(\log n)^{2i}$ many blocks such that $U \cup R$ is disjoint to $\bigcup_j V(M_i(B_j)) \cup \bigcup_j C(B_j)$. Then, there is a maximal rainbow matching $M'$ in $G$ which avoids vertices in $U$ and colours in $R$, and such that $(M_i \setminus M_i(\mathcal{F}_i)) \cup \bigcup_j M_i(B_j) \subseteq M'$.
\vspace{0.2cm}

\noindent To finish the proof, we will use this property to show that for each $i \leq \sqrt{\log n}$, if property $P_i$ holds after step $i$ is completed, then this implies that step $i+1$ can be successfully done and after it is completed, property $P_{i+1}$ holds. Note trivially that property $P_0$ holds after step 0 (taking $M'=M$ always). Take now some $i \leq \sqrt{\log n}$ and assume that property $P_i$ holds after step $i$ is completed. Recall from the description of step $i+1$, that all that is needed to ensure that it can be done, is that the first two options of Lemma \ref{stepprocedure} do not hold for the graph $G_i$, the rainbow matching $M_i$, the family $\mathcal{F}_i$ and the colours $C_{i-1}$. The first option not holding follows from Corollary \ref{Mprimelemma}. Indeed, note that in the graph $G_i$, each colour is a disjoint union of non-trivial cliques with at least $2n+2m+n^{3/4} - 2|M \setminus M_i|$ vertices. This follows since precisely by construction, at most $|M \setminus M_i|$ vertices of $G$ have been deleted up to step $i$. Hence, $G_i$ is a $(n,2(|M_i|+|C_0|)+2m+n^{3/4},m)$-multigraph and so, we can apply Corollary \ref{Mprimelemma} to the matchings $M_i(\mathcal{F}_i) \subseteq M_i$ in order to ensure that there are at least $|C_0|^{1/4} > n^{1/5} \log n$ many edge-disjoint matchings of size $|M_i(\mathcal{F}_i)| + |C_0| - n^{3/4} \geq |M_i(\mathcal{F}_i)| + \frac{1000n}{\sqrt{\log n}}$ using colours in $C_0 \cup C(\mathcal{F}_i) \subseteq C_{i-1} \cup C(\mathcal{F}_i)$ and edges not contained in $V(M_i \setminus M_{i}(\mathcal{F}_i))$. For the second option of Lemma \ref{stepprocedure}, suppose there is such an edge $e = xy$ of some colour $c$ such that $x,y$ belong to two distinct blocks in $\mathcal{F}_i$ and $c \in C_{i-1} \cup C(\mathcal{F}_i)$ does not belong to any of the blocks containing $x,y$. Then, applying property $P_i$ with $U=\{x,y\}, R=\{c\}$ implies that there is a maximal rainbow matching $M'$ in $G$ avoiding the vertices $x,y$ and the colour $c$. But then $M' \cup \{e\}$ contradicts the maximality of $M'$. Concluding we now know that step $i+1$ can be done and goes as described earlier. To finish, we now show the following.
\begin{lem*}
Property $P_{i+1}$ holds after step $i+1$ is completed.
\end{lem*}
\begin{proof}
Naturally, we divide the proof into two cases. First, let us suppose that after step $i$, when Lemma \ref{stepprocedure} was applied, Option 3 held. The reader might want to refer back to the description of our process in order to recall how step $i+1$ goes in this case. Let then $U \subseteq V(\mathcal{F}_{i+1}) = V'_{i+1} \cup V(\mathcal{F}_{i})$ and $R \subseteq C_{i} \cup C(\mathcal{F}_{i+1}) = C_{i-1} \cup C(\mathcal{F}_i)$ be of size at most $n^{1/10}/(\log n)^{2i+2}$ and such that no two members of $U \cup R$ belong to the same block in $\mathcal{F}_{i+1}$. Let also $B_1, B_2, \ldots \in \mathcal{F}_{i+1}$ be a collection of at most $n^{1/10}/(\log n)^{2i+2}$ many blocks such that $U \cup R$ is disjoint to $\bigcup_j V(M_{i+1}(B_j)) \cup \bigcup_j C(B_j) = \bigcup_j V(M_{i}(B_j)) \cup \bigcup_j C(B_j)$. We now check that we can find a maximal rainbow matching $M'$ which ensures the validity of property $P_{i+1}$.

In order to do so, let us first repeat here the important characteristics of the vertices in $V'_{i+1}$ -  these are vertices $v$ for which there are at least $n^{1/10}$ many distinct $x \in V(\mathcal{F}_i)$ such that the edges $m(v)x$ have distinct colours in $C_{i-1} \cup C(\mathcal{F}_i)$; moreover, for each $x$, that colour of $m(v)x$ does not belong to the block in $\mathcal{F}_i$ which contains $x$. Now, since $i \leq \sqrt{ \log n}$, we have \begin{equation}\label{eq1}
n^{1/10} > 10 \cdot (|U|+|R|) \cdot (2 \cdot  4^{i}+1) \geq 10 \cdot (|U|+|R|) \cdot \max_{B \in \mathcal{F}_{i}} |B|
\end{equation} and thus, notice that these characteristics allow us to find a collection of distinct vertices $\{w_u : u \in U \cap V'_{i+1}\} \subseteq V(\mathcal{F}_i) \setminus U$ and distinct colours $\{c_u : u \in U \cap V'_{i+1}\} \subseteq (C_{i-1} \cup C(\mathcal{F}_i)) \setminus R$ with the following properties:
\begin{enumerate}
    \item Each edge $m(u)w_u$ is $c_u$-coloured.
    \item No two members of $\{w_u : u \in U \cap V'_{i+1}\} \cup \{c_u : u \in U \cap V'_{i+1}\}$ belong to the same block in $\mathcal{F}_i$.
    \item No colour $c_u$ or vertex $w_u$ belongs to the same block in $\mathcal{F}_i
    $ that a member of $(U \setminus V'_{i+1}) \cup R$ belongs to.    
\end{enumerate}
Indeed, since each block in $\mathcal{F}_i$ is small enough so that (\ref{eq1}) occurs, we can choose the elements $c_u,w_u$ greedily. Moreover, we are able to ensure the second property in its full generality since the characteristics of the vertices in $V'_{i+1}$ allow us to have, for each $u$, that $w_u$ and $c_u$ do not belong to the same block.

Let then $U' := (U \setminus V'_{i+1}) \cup \{w_u : u \in U \cap V'_{i+1}\} \subseteq V(\mathcal{F}_i)$ and $R'  := R \cup \{c_u : u \in U \cap V'_{i+1}\}\subseteq C_{i-1} \cup C(\mathcal{F}_i)$. Note that from the properties listed above, these are such that no two members of $U' \cup R'$ belong to the same block in $\mathcal{F}_i$. Moreover, both these sets have size at most $2(|U|+|R|) < n^{1/10}/(\log n)^{2i}$, and so, property $P_i$ holding after step $i$ ensures that there exists a maximal rainbow matching $M''$ in $G$ which avoids $U'$, $R'$ and with $$(M_i \setminus M_i(\mathcal{F}_i)) \cup \bigcup_{j : B_j \in \mathcal{F}_i} M_i(B_j) \subseteq M'' .$$Now define the rainbow matching $M' := (M'' \setminus \{u m(u) : u \in U \cap V'_{i+1}\}) \cup \{m(u)w_u: u \in U \cap V'_{i+1}\}$, with naturally, each edge $m(u)w_u$ being assigned the colour $c_u$. We check that it in fact ensures the validity of property $P_{i+1}$. First, note indeed that it is a matching since $M''$ avoids all vertices $w_u$ and moreover, each edge $um(u)$ with $u \in V'_{i+1}$ belongs to $M_i \setminus M_i(\mathcal{F}_i)$ and thus, to $M''$. It is also clearly rainbow because of the previous choice of distinct colours $c_u$, which $M''$ avoids. Further, we trivially have $|M'| \geq |M''|$ and so it is maximal. Recall also that the vertices $w_u$ do not belong to $U$ and the colours $c_u$ do not belong to $R$. Therefore, since $M''$ avoids $U' \cup R'$, this implies that $M'$ indeed avoids $U \cup R$. Finally, we check that $M_{i+1} \setminus M_{i+1}(\mathcal{F}_{i+1}) \subseteq M'$ and $M_{i+1}(B_j) \subseteq M'$ for each $j$. The former holds because $M_{i+1} \setminus M_{i+1}(\mathcal{F}_{i+1})$ is contained in $M_i \setminus M_i(\mathcal{F}_i) \subseteq M''$ and does not contain any edge $v m(v)$ with $v \in V'_{i+1}$, which are precisely the only type of edges we remove from $M''$ to form $M'$. For the latter, note that in this case of Option 3 holding, for each $j$
with $B_j \in \mathcal{F}_i$ we have that $M_{i+1}(B_j)$ is equal to $M_i(B_j)$, which in turn is contained in $M''$. Since $B_j \in \mathcal{F}_i$ means that no edge $v m(v)$ with $v \in V'_{i+1}$ belongs to $M_i(B_j)$, we must also have that $M_i(B_j)$ is contained in $M'$. On the other hand, if $B_j \notin \mathcal{F}_i$ then $B_j$ consists of a singleton set $\{v\}$ for some vertex $v \in V'_{i+1}$ which is then outside $M_{i+1}$ and so, $M_{i+1}(B_j) = \emptyset$.

Now, let us suppose that Option 4 held and let $U \subseteq V(\mathcal{F}_{i+1})$ and $R \subseteq C_i \cup C(\mathcal{F}_{i+1})$ be of size at most $n^{1/10}/(\log n)^{2i+2}$ and such that no two members in $U \cup R$ belong to the same block in $\mathcal{F}_{i+1}$. Let also $B_1, B_2, \ldots \in \mathcal{F}_{i+1}$ be a collection of at most $n^{1/10}/(\log n)^{2i+2}$ many blocks such that $U \cup R$ is disjoint to $\bigcup_j V(M_{i+1}(B_j)) \cup \bigcup_j C(B_j) = \bigcup_j V(M_{i}(B_j)) \cup \bigcup_j C(B_j)$. First, recalling what occurs in this case, our graph and rainbow matching remain the same, i.e., $M_{i+1} = M_i$ and $G_{i+1} = G_i$, and we only add some vertices to the present family of blocks resulting in a family $\mathcal{F}_{i+1} \succeq \mathcal{F}_{i}$ of disjoint $(C_i, 4^{i+1},(i+1) \sqrt{\log n}/100)$-blocks. Furthermore, notice that since every vertex outside $M_i$ already belongs to a block in $\mathcal{F}_i$ (and so, it has the maximal number of blocks possible, because each block must contain a unique vertex outside $M_i$), we can say that the family $\mathcal{F}_i$ consists of blocks $B' \subseteq B$ for each $B \in \mathcal{F}_{i+1}$. 

Next, each $u \in U \setminus V(\mathcal{F}_i)$ belongs to a unique block in $\mathcal{F}_{i+1}$, which we denote as $B_u$. By definition of this block, there then exists a $M_{i+1}(B_u)-C_i-\ldots -M_{i+1}(B_u)-C_i$ path $P_u$ of length at most $(i+1) \sqrt{\log n}/100$, starting at $u$ and ending at $v_u := v_{B_u} = v_{B'_u}$ (recall this is the vertex of the block which is contained outside $M_i$), whose $C_i$-edges are each repeated in at least $n^{1/10}$ many colours of $C_i$. Similarly, for each colour $c \in R \setminus C_i$, belonging to a unique block $B_c \in \mathcal{F}_{i+1}$, there exists a $m_c-C_i-\ldots -M_{i+1}(B_c)-C_i$ path $P_c$ of length at most $(i+1)\sqrt{\log n}/100$ starting at the edge of $M_{i+1}$ of colour $c$, which we denote by $m_c$, ending at $v_c := v_{B_c} = v_{B'_c}$, and whose $C_i$-edges are each repeated in at least $n^{1/10}$ many colours of $C_i$. Recall also that one of the original assumptions is that no two members of $U \cup R$ belong to the same block in $\mathcal{F}_{i+1}$. Therefore, this is preserved to the family $\mathcal{F}_i$ in the following sense: the blocks $B'_u,B'_c \in \mathcal{F}_i$ are all distinct and distinct to the blocks in $\mathcal{F}_i$ that the vertices in $U \cap V(\mathcal{F}_i)$ and the colours in $R \cap C_i$ belong to. Now, since $i \leq \sqrt{\log n}$ and thus \begin{equation}\label{eq}
n^{1/10} > 2 (|U| + |R|) \cdot (i+1) \sqrt{\log n}/100 \cdot 4^{i} \geq \left(2|U| + 2|R| + \sum_u |P_u| + \sum_c |P_c| \right) \cdot \max_{B \in \mathcal{F}_i} |C(B)| ,
\end{equation}we can then greedily pick distinct $C_i$-colours for the $C_i$-edges in the paths $P_u, P_c$ so that the set composed by these colours, which we denote as $C^{*} \subseteq C_i$, is rainbow and has the following properties:
\begin{enumerate}
    \item No two members of the set $C^{*}$ belong to the same block in $\mathcal{F}_i$.
    \item No member of $C^{*}$ belongs to a block $B'_u$ with $u \in U$ or a block $B'_c$ with $c \in R$.
    \item No member of $C^{*}$ belongs to one of the blocks $B'_1 ,B'_2 , \ldots $.
\end{enumerate}
\noindent Given these, let now $U' := (U \cap V(\mathcal{F}_i)) \cup \{v_c : c \in R \setminus C_i \} \cup \{v_u : u \in U \setminus V(\mathcal{F}_i)\}$ and $R' := (R \cap C_i) \cup C^{*}$. Notice that by the three properties above, no two members of $U' \cup R'$ belong to the same block in $\mathcal{F}_i$. Also, since the vertices $v_u = v_{B_u}, v_c = v_{B_c}$ with $u \in U \setminus V(\mathcal{F}_i)$ and $c \in R \setminus C_i$ are all contained outside $M_i$, they are disjoint to the sets $V(M_i(B'_u)), V(M_i(B'_c)), V(M_i(B'_j))$. Therefore, $U' \cup R'$ is disjoint to these sets as well as to the sets $C(B'_u), C(B'_c), C(B'_j)$ (because of the properties of $C^{*}$ above). Therefore, we can apply property $P_i$ since further, $U'$ and $R'$ are both of size at most $2|U| + 2|R| + \sum_u |P_u| + \sum_c |P_c| \leq (|U|+|R|) (2 + (i+1) \sqrt{\log n}/100) \leq n^{1/10}/(\log n)^{2i}$. 

Then, property $P_i$ holding after step $i$ ensures that there exists a maximal rainbow matching $M''$ in $G$ which avoids $U'$, $R'$ and with $$(M_i \setminus M_i(\mathcal{F}_i)) \cup \bigcup_j M_i(B'_j) \cup \bigcup_{u} M_i(B'_u) \cup \bigcup_{c} M_i(B'_c) \subseteq M'' .$$
We claim that we can now use all the disjoint paths $P_u,P_c$ to form a maximal rainbow matching $M'$ with the desired properties ensuring the validity of property $P_{i+1}$. Indeed, note first that for each $u \in U \setminus V(\mathcal{F}_i)$, the matching $M''$ avoids the vertex $v_u$ and uses all the edges in $M_{i+1}(B_u)$ since $M_{i+1}(B_u) \subseteq (M_i \setminus M_i(\mathcal{F}_i)) \cup M_i(B'_u)$. The equivalent happens for the colours $c \in R \setminus C_i$. Therefore, since also $M''$ avoids the colours in $C^{*}$, we can form a maximal rainbow matching $M'$ by substituting the edges of $M''$ used in the paths $P_u,P_c$ by the rest of the edges in these paths (which we have already picked distinct colours for when constructing the set $C^{*}$). Note now that because of this construction, $M'$ avoids all the vertices $u \in U \setminus V(\mathcal{F}_i)$ and all the colours $c \in R \setminus C_i$, as well as the vertices in $U \cap V(\mathcal{F}_i)$ and colours in $R \cap C_i$. Further, all the edges in $M_{i+1} \setminus M_{i+1}(\mathcal{F}_{i+1})$ belong to $M'$ since this set is contained in $M_i \setminus M_i(\mathcal{F}_i) \subseteq M''$ and is disjoint to the paths $P_u, P_c$. Also, to conclude, all the edges in the matchings $M_{i+1}(B_1), M_{i+1}(B_2), \ldots$ are in $M'$ since each set $M_{i+1}(B_j)$ is contained in $(M_i \setminus M_i(\mathcal{F}_i)) \cup M_i(B'_j)$ and is disjoint to the paths $P_u,P_c$. The latter indeed occurs because by the initial assumption, no member of $U \cup R$ belongs to $V(M_i(B'_j))$.
\end{proof}
\subsection{A lower bound}
As we indicated in the introduction, to finish our study of the Grinblat multiplicity problem, we give the following construction.
\begin{prop}
Let $d$ be an integer and $n > 10d^3 \log d$ such that $d | n-1$. Then, there exists a $(n,\left(2+1/d \right)(n-1), n/2d + O\left(n/d^2\right))$-multigraph with no matching of size $n$. 
\end{prop}
\begin{proof}
Let $H$ be a $n$-edge-coloured multigraph on the vertex set $\{1, \ldots, 2d+1\}$ constructed by doing the following: independently for each colour $c$, pick uniformly at random a spanning subgraph $H_c \subseteq H$ consisting of a disjoint union of $d-1$ edges and one triangle; set the edges of colour $c$ to be the edges of $H_c$. Note that the multiplicity of each edge $e \in H$ behaves like a $\textsc{Bin}(n,p)$ random variable with $p := \frac{d+2}{d(2d+1)}$. Therefore, since there are $O(d^2)$ edges, \textit{whp} the multiplicity of $H$ is at most $n/2d + O\left(n/d^2\right)$. Now, let $G$ be the disjoint union of $\frac{n-1}{d}$ copies of $H$. Since $|H| = 2d+1$, it has no matching of size $d+1$ and thus, $G$ has no matching of size $n$. Moreover, by construction of $H$, each colour has at least $\frac{n-1}{d} \cdot (2d+1)$ vertices in its colour class.
\end{proof}
Note that for multiplicity $m = \eps n$ with $\eps > n^{-1/3 + o(1)}$, this gives an example of $(n,2n+2\eps n-O(\eps^2 n),\eps n)$-multigraphs without the desired rainbow matching. This shows that the error term $2m$ in Theorem \ref{simplegrinblatthm} is asymptotically tight. 
\section{Concluding remarks}
In this paper we obtained improved bounds for a wide variety of rainbow matching problems, resolving several conjectures. For this, we introduced an effective method for proving strong asymptotic results when their weak versions are known. The most natural open problem is to obtain even better error terms in all the problems considered here.

For the Aharoni-Berger conjecture(s), we now have polynomial bounds on the error term in both the strong and weak asymptotic versions. This is quite far from the best bounds in Ryser's conjecture, the main problem motivating the Aharoni-Berger conjecture, where we know (see \cite{keevash2020new}) how to find rainbow matchings of size $n-O(\log n/\log\log n)$. It would then be interesting to prove sub-polynomial bounds for the Aharoni-Berger conjecture(s) as well as for the other problems considered in this paper. Note that our sampling trick currently requires error terms to be polynomial. Thus, having a weak asymptotic result with sub-polynomial error term will not immediately imply a sub-polynomial error term for the corresponding strong asymptotic version.

For Alspach's conjecture, we can actually prove a sub-polynomial bound in the weak asymptotic.
\begin{prop}
Let $G$ be a $2$-factorized graph with $n$ colours. Then there is a rainbow matching of size $n-O(\log n / \log\log n)$.
\end{prop}
\begin{proof}[  sketch]
By Theorem~\ref{alspachthm}, we can assume that $N:=|G|\leq 2n(1+n^{-0.24})$. This means that $G$ is essentially a complete graph. Let the vertices of $G$ be $v_1, \dots, v_N$. Randomly orient the graph so that each colour is a union of directed cycles (for each cycle choosing its direction independently). Randomly partition $V(G)$ into two sets $X,Y$ with each vertex ending up in each set with probability $1/2$. For each colour $c$ we 
delete all edges which aren't directed from $X$ to $Y$.  Call the resulting graph $H$. 

Using standard probabilistic arguments, one can show that for some $\eps>0$, with positive probability $n\leq |X|,|Y|\leq n+n^{1-\eps}$, all vertices have $d_H(v)=n/2\pm n^{1-\eps}$, $d_H(u,v)\leq n/4+ n^{1-\eps}$, and all colours $c,c'$ have $e_H(c)=n/2\pm n^{1-\eps}$ and $a_H(c,c'), b_H(c,c')= n/4+n^{1-\eps}$ (where $a_H(c,c')/b_H(c,c')$ denote the number of vertices in $A/B$ incident to edges of both colours $c$ and $c'$ in $H$).
 Next we apply Corollary~4.6 from~\cite{keevash2020new} which essentially says that graphs with these properties contain a rainbow matching of size $n-O(\log n/\log \log n)$ (actually Corollary~4.6 has the slightly stronger assumption $d_H(u,v)= n/4\pm n^{1-\eps}$ on the graph; however it can be checked that this assumption is never fully used in the proof, and just the upper bound suffices on these quantities). 
\end{proof}
 
\vspace{0.4cm}
\noindent
{\bf Acknowledgements.} The authors would like to thank He Guo for carefully reading the manuscript and pointing out some inaccurate details.


\begin{thebibliography}{10}

\bibitem{aharoni2009rainbow}
R.~Aharoni and E.~Berger.
\newblock Rainbow matchings in $r$-partite $r$-graphs.
\newblock {\em Electron. J. Combin.}, 16(1):R119, 2009.

\bibitem{aharoni2019large}
R.~Aharoni, E.~Berger, M.~Chudnovsky, S.~Zerbib.
\newblock On the Aharoni-Berger conjecture for general graphs.
\newblock {\em arXiv:2012.14992}, 2020.

\bibitem{aharoni2015generalization}
R.~Aharoni, P.~Charbit, and D.~Howard.
\newblock On a Generalization of the Ryser-Brualdi-Stein Conjecture.
\newblock {\em J. Graph Theory}, 78(2):143--156, 2015.

\bibitem{aharoni2017representation}
R.~Aharoni, D.~Kotlar, and R.~Ziv.
\newblock Representation of large matchings in bipartite graphs.
\newblock {\em SIAM J. Discrete Math.}, 31(3):1726--1731, 2017.

\bibitem{AKS}
N.~Alon, J.-H. Kim, and J.~Spencer.
\newblock Nearly perfect matchings in regular simple hypergraphs.
\newblock {\em Israel J. Math.}, 100(1):171--187, 1997.

\bibitem{alspach1988problem}
B.~Alspach.
\newblock Problem 89.
\newblock {\em Discrete Math.}, 69:106, 1988.

\bibitem{alspach1992orthogonal}
B.~Alspach, K.~Heinrich, and G.~Liu.
\newblock Orthogonal factorizations of graphs,
\newblock {in: \em Contemporary Design Theory: A Collection of Surveys}, Wiley, New York, 13--40, 1992.

\bibitem{anstee1998orthogonal}
R.~P. Anstee and L.~Caccetta.
\newblock Orthogonal matchings.
\newblock {\em Discrete Math.}, 179(1-3):37--47, 1998.

\bibitem{barat2017rainbow}
J.~Bar{\'a}t, A.~Gy{\'a}rf{\'a}s, and G.~N. S{\'a}rk{\"o}zy.
\newblock Rainbow matchings in bipartite multigraphs.
\newblock {\em Period. Math. Hungar.}, 74(1):108--111, 2017.

\bibitem{brualdi1991combinatorial}
R.~A. Brualdi, H.~J. Ryser, et~al.
\newblock {\em Combinatorial matrix theory}, volume~39.
\newblock Springer, 1991.

\bibitem{caccetta1992premature}
L.~Caccetta and S.~Mardiyono.
\newblock Premature sets of one-factors.
\newblock {\em Australas. J. Comb.}, 5:229--252, 1992.

\bibitem{clemens2015improved}
D.~Clemens and J.~Ehrenm{\"u}ller.
\newblock An improved bound on the sizes of matchings guaranteeing a rainbow matching.
\newblock {\em Electron. J. Combin.}, 23(2):P2.11, 2016.

\bibitem{clemens2017sets}
D.~Clemens, J.~Ehrenm{\"u}ller, and A.~Pokrovskiy.
\newblock On sets not belonging to algebras and rainbow matchings in graphs.
\newblock {\em J. Combin. Theory Ser. B}, 122:109--120, 2017.

\bibitem{correia2020full}
D.~M. Correia and L.~Yepremyan.
\newblock Full rainbow matchings in equivalence relations.
\newblock {\em arXiv:2002.08974}, 2020.

\bibitem{hoeffding}
D.~Dubhashi and A.~Panconesi.
\newblock Concentration of measure for the analysis of randomised algorithms.
\newblock {\em Cambridge University Press}, 2009.

\bibitem{gao2017full}
P.~Gao, R.~Ramadurai, I.~Wanless, and N.~Wormald.
\newblock Full rainbow matchings in graphs and hypergraphs.
\newblock {\em Combin. Probab. Comput.}, 30(5):762--780, 2021.

\bibitem{grinblat2002algebras}
L.~{\v{S}}. Grinblat.
\newblock {\em Algebras of sets and combinatorics}.
\newblock American Mathematical Society, 2002.

\bibitem{grinblat2015families}
L.~{\v{S}}. Grinblat.
\newblock Families of sets not belonging to algebras and combinatorics of
  finite sets of ultrafilters.
\newblock {\em J. Inequal. Appl.}, 2015(1):1--19, 2015.

\bibitem{keevash2020new}
P.~Keevash, A.~Pokrovskiy, B.~Sudakov, and L.~Yepremyan.
\newblock New bounds for Ryser's conjecture and related problems.
\newblock {\em arXiv:2005.00526}, 2020.

\bibitem{keevash2018rainbow}
P.~Keevash and L.~Yepremyan.
\newblock Rainbow matchings in properly colored multigraphs.
\newblock {\em SIAM J. Discrete Math.}, 32(3):1577--1584, 2018.

\bibitem{kostochka}
A.~V. Kostochka and V.~R{\"o}dl.
\newblock Partial steiner systems and matchings in hypergraphs.
\newblock {\em Random Structures Algorithms}, 13(3-4):335--347, 1998.

\bibitem{kotlar2014large}
D.~Kotlar and R.~Ziv.
\newblock Large matchings in bipartite graphs have a rainbow matching.
\newblock {\em European J. Combin.}, 38:97--101, 2014.

\bibitem{kouider1989existence}
M.~Kouider and D.~Sotteau.
\newblock On the existence of a matching orthogonal to a 2-factorization.
\newblock {\em Discrete Math.}, 73(3):301--304, 1989.

\bibitem{nivasch2017rainbow}
G.~Nivasch and E.~Omri.
\newblock Rainbow matchings and algebras of sets.
\newblock {\em Graphs Combin.}, 33(2):473--484, 2017.

\bibitem{pokrovskiy2015rainbow}
A.~Pokrovskiy.
\newblock Rainbow matchings and rainbow connectedness.
\newblock {\em Electron. J. Combin.}, 24(1):P1.13, 2017.

\bibitem{pokrovskiy2018approximate}
A.~Pokrovskiy.
\newblock An approximate version of a conjecture of Aharoni and Berger.
\newblock {\em Adv. Math.}, 333:1197--1241, 2018.

\bibitem{qu2015orthogonal}
C.~Qu, G.~Wang, and G.~Yan.
\newblock Orthogonal matchings revisited.
\newblock {\em Discrete Math.}, 338(11):2080--2088, 2015.

\bibitem{ryser1967neuere}
H.~J. Ryser.
\newblock Neuere probleme der kombinatorik.
\newblock {\em Vortr{\"a}ge {\"u}ber Kombinatorik, Oberwolfach}, 69:91, 1967.

\bibitem{stein1975transversals}
S.~K. Stein.
\newblock Transversals of latin squares and their generalizations.
\newblock {\em Pacific J. Math.}, 59(2):567--575, 1975.

\bibitem{stong2002orthogonal}
R.~Stong.
\newblock Orthogonal matchings.
\newblock {\em Discrete Math.}, 256(1-2):513--518, 2002.

\bibitem{woolbright1978n}
D.~E. Woolbright.
\newblock An $n \times n$ latin square has a transversal with at least $n-\sqrt{n}$ distinct symbols.
\newblock {\em J. Combin. Theory Ser.  A}, 24(2):235--237, 1978.

\end{thebibliography}
\end{document}